\documentclass[11pt]{article}
\usepackage[margin=1in]{geometry}
\usepackage{graphicx,xcolor}
\usepackage{amssymb,amsmath,amsthm,mathrsfs,paralist,bm,esint,setspace}
\RequirePackage[colorlinks,citecolor=blue,urlcolor=blue]{hyperref}

\newcommand{\R}{{\mathbb{R}}}
\newcommand{\E}{\mathrm{E}}
\renewcommand{\P}{\mathrm{P}}
\renewcommand{\d}{\mathrm{d}}
\newcommand{\e}{\mathrm{e}}
\newcommand{\Var}{\text{\rm Var}}

\input cyracc.def

\title{Spatial ergodicity for SPDEs via a Poincar\'e-type inequality\thanks{Research supported in part by the NSF grant DMS-1608575.}}
\author{Le Chen\thanks{Email: \texttt{le.chen@unlv.edu}}\\University of Nevada, Las Vegas\\
	\and
	Davar Khoshnevisan\thanks{Email: \texttt{davar@math.utah.edu}}\\University of Utah\\
	\and
	Fei Pu\thanks{Email: \texttt{pu@math.utah.edu}}\\University of Utah}
\date{May 16, 2019}

\begin{document}
\newtheorem{stat}{Statement}[section]
\newtheorem{proposition}[stat]{Proposition}
\newtheorem*{prop}{Proposition}
\newtheorem{corollary}[stat]{Corollary}
\newtheorem{theorem}[stat]{Theorem}
\newtheorem{lemma}[stat]{Lemma}
\theoremstyle{definition}
\newtheorem{definition}{Definition}
\newtheorem*{cremark}{Remark}
\newtheorem{remark}[stat]{Remark}
\newtheorem*{OP}{Open Problem}
\newtheorem{example}[stat]{Example}
\newtheorem{nota}[stat]{Notation}
\numberwithin{equation}{section}
\maketitle

\begin{abstract}
	Consider a parabolic stochastic PDE of the form
	$\partial_t u=\frac12\Delta u + \sigma(u)\eta$, where
	$u=u(t\,,x)$ for $t\ge0$ and $x\in\R^d$, $\sigma:\R\to\R$
	is Lipschitz continuous and non random, and $\eta$ is a centered Gaussian
	noise that is white in time and colored in space, with a possibly-signed
	homogeneous spatial correlation function $f$. If, in addition, 
	$u(0)\equiv1$, then we prove that, under a mild
	decay condition on $f$, the process $x\mapsto u(t\,,x)$ is stationary and ergodic at
	all times $t>0$.  It has been argued that, when coupled with
	moment estimates,  spatial ergodicity of $u$
	teaches us about the intermittent nature of the solution to such
	SPDEs \cite{BertiniCancrini1995,KhCBMS}. Our results provide rigorous justification of
	of  such discussions. The proof rests on novel facts about functions
	of positive type, and on strong localization bounds for comparison of SPDEs.
\end{abstract}

\section{Introduction}

The principal aim of this article is to establish relatively simple-to-check, but also
broad, conditions under which the solution
$u=\{u(t\,,x)\}_{t\ge0\,,x\in\R^d}$ to a parabolic stochastic PDE is
\emph{spatially stationary and ergodic}.
Equivalently, we would like to know conditions under which
$u(t)$ is stationary and ergodic, in its spatial variable $x$, at all times $t>0$. This problem,
and its relation to intermittency, have been
mentioned informally  for example
in the introduction of Bertini and Cancrini \cite{BertiniCancrini1995}; see also Chapter
7 of Khoshnevisan \cite{KhCBMS}. This problem is also connected somewhat loosely to novel
applications of Malliavin calculus to central limit theorems
for parabolic SPDEs; see Huang et al \cite{HuangNualartViitasaari2018,HuangNualartViitasaariZheng2019}.

In order for spatial ergodicity to be a meaningful property, one needs to consider parabolic SPDEs
for which the solution is {\it a priori} a stationary process in its spatial variable. Thus, we study the
following archetypal parabolic problem:
\begin{equation}\label{SHE}
\begin{cases}
\displaystyle
\partial_t  u = \frac12\Delta u + \sigma(u)\eta
		&\text{on $(0\,,\infty)\times\R^d$},\\[0.5em]
\displaystyle	u(0)\equiv 1,
\end{cases}\end{equation}
where $\sigma$ is Lipschitz continuous and non random, and $\eta$
denotes a generalized, centered, Gaussian random field with
covariance form
\[
	\E\left[\eta(t\,,x)\eta(s\,,y)\right] = \delta_0(t-s) f(x-y)
	\qquad\text{for all $s,t\ge0$ and $x,y\in\R^d$},
\]
where $f$ is a nonnegative-definite locally-integrable function  on $\R^d$.
Somewhat more formally, the Wiener-integral process
$\psi \mapsto\eta(\psi) :=\int_{\R_+\times\R^d}
 \psi(t\,,x)\,\eta(\d t\,\d x)$ is linear a.s.\ and satisfies
\[
	\text{\rm Cov}\left( \eta( \psi _1)\,,\eta( \psi _2)\right) = \int_0^\infty
	\left\langle  \psi_1(t)\,, \psi_2(t)*f\right\rangle_{L^2(\R^d)}\d t,
\]
for every $\psi_1,\psi_2$ in
the space $C_c(\R_+\times\R^d)$ of all compactly-supported, continuous,
real-valued functions on $\R_+\times\R^d$.

In order to simplify our exposition, we consider
throughout only the case that
\begin{equation}\label{f=h*h}
	f = h*\tilde{h},
\end{equation}
for a possibly-signed function  $h:\R^d\to\R$, where $\tilde{h}(x):=h(-x)$ defines
the reflection of $h$, and the function $h$ is assumed to
has enough regularity to ensure among other things
that the  convolution in \eqref{f=h*h} is well defined in the sense of distributions.
Condition \eqref{f=h*h} is enforced throughout the paper, and without further mention.
In this case, we have the elegant formula,
\[
	\Var\left( \eta( \psi )\right) = \int_0^\infty
	\left\| \psi(t)*h\right\|_{L^2(\R^d)}^2\d t,
	\qquad\text{
	valid for all $\psi\in C_c(\R_+\times\R^d)$.}
\]

The solution theory for \eqref{SHE} is particularly well established when $f$
is a non-negative, non-negative definite, and tempered function on $\R^d$. In that case,
we know from classical harmonic analysis that the Fourier transform
$\widehat{f}$ of $f$ is a tempered Borel measure on $\R^d$. In that setting,
the theory of Dalang \cite{Dalang1999} implies that, if in addition,
\begin{equation}\label{Dalang}
	h\in H_{-1}(\R^d),\quad\text{equivalently,}\quad \int_{\R^d}
	\frac{|\widehat{h}(x)|^2}{1+\|x\|^2}\,\d x<\infty,
\end{equation}
then \eqref{SHE} has a random-field solution $u$ that is unique subject to
the integrability condition
\begin{equation}\label{cond:moment}
	\sup_{t\in[0,T]}\sup_{x\in\R^d}\E\left(|u(t\,,x)|^k\right)<\infty
	\qquad\text{for every $T>0$ and $k\ge2$},
\end{equation}
and $u$ is both $L^k(\P)$-continuous in $(t\,,x)$, and weakly stationary in $x$
for every $t>0$.
Furthermore, it is known that Condition \eqref{Dalang} is
necessary and sufficient for example when $\sigma$ is a non-zero constant;
see Dalang \cite{Dalang1999}, as well as Peszat and Zabczyk \cite{PeszatZabczyk2000}.

There is also a literature on well-posedness and regularity theory for \eqref{SHE}
when the spatial correlation function $f$ is signed, though such results tend to be
applicable in a more specialized setting
as compared with the
theory of Dalang \cite{Dalang1999}; see for example
\cite{ChenHuNualart2017,ChenHuang2019,%
HuHuangLeNualartTindel2017,HuHuangNualartTindel2015,%
HuangLeNualart2017}.

Here, we prove that a mild integrability condition on $h$ implies that $|h|\in H_{-1}(\R^d)$
(see \eqref{cond:omega} and Lemma \ref{lem:Dalang:Lp}), which in turn implies the existence of a
spatially stationary random-field solution $u$ to \eqref{SHE} that is unique
subject to \eqref{cond:moment} (see Theorem \ref{th:exist}).
More significantly, we prove that the ensuing Condition \eqref{cond:omega}
on $h$ ensures that $u$
is spatially ergodic. In any case, the end result is the following.\footnote{%
	For a very brief
	discussion of relevant measurability issues, see Remark \ref{rem:doob} below.}

\begin{theorem}\label{th:main:intro}
	Choose and fix a real number $p>1$ and let $q:=p/(p-1)$.
	Suppose that $h\in L^p_{\text{\it loc}}(\R^d)$ satisfies
	\begin{equation}\label{cond:omega}
		\int_0^1 \left( \|h\|_{L^p(\mathbb{B}_r)} \|h\|_{L^q(\mathbb{B}_r^c)}
		+ \|h\|_{L^2(\mathbb{B}_r^c)}^2\right)\omega_d(r)\,\d r<\infty,
	\end{equation}
	where, for every $r>0$,
	\begin{equation}\label{eq:omega}
		\omega_d(r) := \begin{cases}
			1&\text{if $d=1$},\\
			r\log_+(1/r)&\text{if $d=2$},\\
			r&\text{if $d\ge3$,}
		\end{cases}
	\end{equation}
	$\log_+(z):=\log(z\vee\e)$ for all $z\in\R$,
	and $\mathbb{B}_s := \{y\in\R^d:\, \|y\|\le s\}$ for every $s>0$.
	Then, given the spatial correlation structure \eqref{f=h*h},
	the SPDE \eqref{SHE} has a spatially stationary and ergodic random-field solution
	$u$ that is unique subject to the integrability condition \eqref{cond:moment}.
\end{theorem}

\begin{remark}
	We have selected the initial data to be identically 1 in Theorem \ref{th:main:intro}
	to be concrete.
	The same method of proof shows that Theorem \ref{th:main:intro}
	continues to hold if the initial data is an arbitrary
	stationary random field $\{u(0\,,x)\}_{x\in\R^d}$ that is independent of $\eta$
	and is continuous in $L^k(\P)$ for every real number $k\ge1$.
\end{remark}

The following is an immediate consequence of Theorem \ref{th:main:intro},
and presents easy-to-check conditions for \eqref{SHE} to have a
unique random-field solution that is spatially
ergodic (as well as stationary).

\begin{corollary}\label{co:main:intro}
	Suppose that $h:\R^d\to\R$ is Borel measurable, and either that $h\in L^2(\R^d)$ or
	that there exist $\alpha\in(0\,,d\wedge 2)$
	and $\beta>0$ such that
	\begin{equation}\label{cond:co:main:intro}
		\sup_{\|w\|<1} \|w\|^{(d+\alpha)/2}|h(w)| <\infty
		\quad\text{and}\quad
		\sup_{\|z\|>1} \|z\|^{(d+\beta)/2}|h(z)| <\infty.
	\end{equation}
	Then, \eqref{SHE} has a random-field solution $u$ that is unique
	subject to the moment condition \eqref{cond:moment}.
	Moreover, $u(t)$ is stationary and ergodic for every $t>0$.
\end{corollary}

It is worth noting that, whereas \eqref{cond:omega} is a global integrability condition on $h$,
\eqref{cond:co:main:intro} involves: (i) A local condition on the behavior of $h$
near the origin; and (ii) A separate local-at-infinity (growth) condition on $h$. Still, it is 
not hard to argue that
\eqref{cond:co:main:intro} implies \eqref{cond:omega};
see \S\ref{sec:Poincare} below.

It is also worth noting that the first
(local) condition on $h$ in \eqref{cond:co:main:intro} is there merely to ensure that
$|h|\in H_{-1}(\R^d)$, which in turn will imply that
\eqref{SHE} has a solution. The second (growth)
condition on $h$ in \eqref{cond:co:main:intro}  is the more interesting hypothesis.
That condition is responsible
for ensuring that $h$ --- whence also $f$ --- decays sufficiently rapidly so that
spatial ergodicity of the solution $u$ to \eqref{SHE} is ensured.
In \S\ref{sec:example} below we construct an example of $f$ for which \eqref{SHE} has a nice
random-field solution that is spatially stationary but not ergodic. This can be done because
$f$ does not have sufficient (in fact, any) spatial decay of correlations.

Our ergodicity result (Theorem \ref{th:main:intro})
is related to the title of the paper because
Theorem \ref{th:main:intro} is a consequence of a Poincar\'e-type inequality for
the occupation measure of $u(t)$; see the paragraph
that follows Theorem \ref{th:Poincare}. Next, we describe a
special case of that Poincar\'e
inequality, presented  in the context of the simple-to-describe Corollary \ref{co:main:intro}.

\begin{corollary}[A Poincar\'e inequality]\label{co:Poincare:intro}
	If \eqref{cond:co:main:intro} holds, then for every $T>0$
	there exists a real number $C_T=C_T(\sigma\,,d\,,h)>0$ such that
	\[
		\sup_{t\in[0,T]}
		\Var\left( \frac{1}{N^d}\int_{[0,N]^d} g(u(t\,,x))\,\d x\right)
		\le  C_T\left[\frac{(\log N)^{3/2}}{N}\right]^{d\beta/(d + \beta)},
	\]
	uniformly for all real numbers $N>1$ and all 1-Lipschitz functions $g:\R\to\R$.
\end{corollary}

Corollary \ref{co:Poincare:intro} is proved in \S\ref{sec:Poincare}.
The main technical result of this paper is in fact
a much more general Poincar\'e-like inequality (Theorem \ref{th:Poincare}).
But that result is more involved, and its
precise description requires some development. Therefore,
the more general Poincar\'e-like inequality will be presented later on in \S\ref{sec:Poincare},
together with its proof.

Consider for the moment the special case $d=1$ and formally
let $\alpha\to0$ and $\beta\to\infty$ in Corollary \ref{co:Poincare:intro}
to deduce --- heuristically, of course --- 
that the corollary ought to cover the important case where $\eta$
denotes space-time white noise [that is, $f=\delta_0$], and that
\[
	\sup_{t\in[0,T]}
	\Var\left( \frac{1}{N}\int_0^N g(u(t\,,x))\,\d x\right)
	\le  C_T\frac{(\log N)^{3/2}}{N},
\]
uniformly for all real numbers $N>1$ and all 1-Lipschitz functions $g:\R\to\R$.
In a separate paper \cite{CKP2019} we intend to prove that this is so,
and that the Poincar\'e constant $O([\log N]^{3/2}/N)$
can be improved upon further. 
That same paper [{\it ibid.}] will also contain a number of more
detailed applications of such inequalities, specialized to the setting of space-time
white noise.

Here is a brief outline of the paper: In \S\ref{sec:example} we present an example
which shows that we cannot expect spatial ergodicity of the solution
of \eqref{SHE} unless $f$ exhibits some sort of decay at infinity (such as the conditions
of Theorem \ref{th:main:intro} on $h$, hence on $f$). Section \ref{sec:FPT} includes
comments and a few harmonic-analytic results on functions of positive type.
Section \ref{sec:Dalang} discusses known results on
the well-posedness of \eqref{SHE}, and discusses how
conditions of Theorem \ref{th:main:intro} ensure among other things that the absolute
value of $h$ is in the classical space Hilbert space $H_{-1}(\R^d)$. In \S\ref{sec:SC}
we extend the stochastic Young inequality of
Walsh integrals \cite{ConusKh2012,FoondunKhoshnevisan2009}
to the case that $f$ is possibly signed
and satisfies the conditions of Theorem \ref{th:main:intro}. It is shown
in \S\ref{sec:wellposedness} that the well-posedness of \eqref{SHE} is a ready consequence
of the mentioned stochastic Young's inequality; see Theorem \ref{th:exist}.
The stationarity
assertion of Theorem \ref{th:main:intro} is proved next in \S\ref{sec:stat}.
Section \ref{sec:localization} contains a technical localization construction
which strengthens and improves an earlier one in Conus et al \cite{CJK2013,CJKS2013}.
That localization procedure
forms the basis of a Poincar\'e inequality that is
presented in \S\ref{sec:Poincare}; see Theorems \ref{th:Poincare} and \ref{th:full}.
Finally,
Theorem \ref{th:main:intro} is proved shortly following the proof of Theorem \ref{th:full}.

Let us close the Introduction with a brief description of the notation of this paper.
Throughout we write ``$g_1(x)\lesssim g_2(x)$ for all $x\in X$'' when there exists a real number
$L$ such that $g_1(x)\le Lg_2(x)$ for all $x\in X$. Alternatively, we might write
``$g_2(x)\gtrsim g_1(x)$
for all $x\in X$.'' By ``$g_1(x)\asymp g_2(x)$ for all $x\in X$'' we mean that
$g_1(x)\lesssim g_2(x)$ for all $x\in X$ and $g_2(x)\lesssim g_1(x)$ for all $x\in X$. Finally,
``$g_1(x)\propto g_2(x)$ for all $x\in X$'' means that there exists a real number $L$
such that $g_1(x)=L g_2(x)$ for all $x\in X$.

Throughout, we write
\[
	\fint_E \psi(x)\,\d x := \frac{1}{|E|}\int_E\psi(x)\,\d x,
\]
whenever $\psi:\R^d\to\R$ is integrable on a Lebesgue-measurable set
$E\subset\R^d$ whose Lebesgue measure $|E|$ is strictly positive.

\section{An non-ergodic example}\label{sec:example}
In the Introduction we alluded that
if the tails of the spatial correlation function
$f$ do not vanish, then we cannot generally expect $u(t)$ to be  ergodic
for all $t\ge0$.
We now describe this in the context of
an  example in which the spatial correlation $f(x)$ does not decay as $\|x\|\to\infty$, the solution $u$
exists and is non-degenerate, and $u$ is
not spatially ergodic at  positive times.

First, we might as well rule out trivialities by assuming that
\begin{equation}\label{sigma(1)neq 0}
	\sigma(1)\neq0.
\end{equation}
Otherwise, one can see easily that $u(t\,,x)\equiv1$; in this case,
$u(t)$ is  ergodic for all $t\ge0$, but only in a vacuous sense.

Next, let us choose and fix a number $\lambda>0$, and suppose that
\begin{equation}\label{f=1}
	f(x)=\lambda^2\qquad\text{for all $x\in\R^d$},
\end{equation}
to ensure that the tails of $f$ do not decay. In this case, standard SPDE theory
shows that we can realize the noise $\eta(\d t\,\d x)$ as $\lambda\,\d W_t\,\d x$, where $W$ is standard,
one-dimensional Brownian motion. Thus, we can see from \eqref{SHE} and standard
arguments that, under
\eqref{f=1},
\begin{equation}\label{u=X}
	u(t\,,x) = X_t\qquad\text{for all $t\ge0$ and $x\in\R^d$ a.s.,}
\end{equation}
where $X$ denotes the unique solution of the one-dimensional It\^o stochastic differential equation,
\[
	\d X_t = \lambda\sigma(X_t)\,\d W_t,
	\qquad
	\text{subject to  $X_0=1.$}
\]
One can deduce from this that and standard estimates that
\[
	\lim_{t\to0^+}
	\frac1t\Var(X_t) =  \lambda^2\sigma^2(1),
\]
whence $\Var(X_t)>0$
for all $t$ small. Thus, we conclude from this and \eqref{u=X} that, under conditions
\eqref{sigma(1)neq 0} and \eqref{f=1}, the process $u(t)$ is not  ergodic for all $t>0$.
In fact, a little more effort shows that $\Var(X_t)>0$ for all $t>0$, thanks to the Markov property.
And this implies that $u(t)$ is not ergodic for any $t>0$.

In this example, the spatial correlation function $f\equiv\lambda^2$,
as given in \eqref{f=1}, does not satisfy \eqref{f=h*h}. However, one would
guess, based on this example, that it should be possible to construct
correlation functions $f$ that do satisfy \eqref{f=h*h}, and yet do not have sufficient
decay at $\infty$ to ensure spatial ergodicity of $u$. It might be interesting to construct such
examples.

\section{Functions of positive type}\label{sec:FPT}

Throughout, we use the following notation for open balls:
\[
	\mathbb{B}_r(x) := \left\{ y\in\R^d:\, \|x-y\|<r\right\}
	\quad\text{and}\quad
	\mathbb{B}_r := \mathbb{B}_r(0)
	\qquad\text{for $x\in\R^d$ and $r>0$}.
\]
A part of this notation was already introduced in the statement of Theorem \ref{th:main:intro}.
With this in mind, let us recall from classical harmonic analysis the following
(see Kahane \cite{Kahane}).

\begin{definition}\label{def:PT}
	A function $g:\R^d\to\R$ is of \emph{positive type} if:
	\begin{compactenum}
		\item $g$ is locally integrable and positive definite in the sense of distributions
			(that is, $\widehat{g}$ is non negative and hence a Borel measure, thanks to
			the Riesz representation theorem); and
		\item The restriction of $g$ to $\mathbb{B}_r^c$ is a uniformly continuous (and
			hence also bounded) function for every $r>0$.\footnote{Some authors insist
			that $g$ is of positive type if, in addition to the requirements of Definition \ref{def:PT},
			$g(0):=\lim_{x\to0}g(x)=\infty$. We do not do that here.}
	\end{compactenum}
\end{definition}

Typical examples include well-known
positive-definite functions such as $g(x)=\exp(-\alpha\|x\|^\beta)$ and/or
$g(x)=(\alpha'+\|x\|^\beta)^{-1}$ for constants
$\alpha\ge0$, $\alpha'>0$, and $\beta\in(0\,,2]$, etc. There
are also unbounded examples such as Riesz kernels
($g(x)=\|x\|^{-\gamma}$ for $\gamma\in(0\,,d)$), as well as products of
the above such as $g(x)=\|x\|^{-\gamma}\exp(-\alpha\|x\|^\beta)$, etc.

The main goal of this section is to present a  family $\cup_{p>1}\mathcal{F}_p(\R^d)$
of real-valued functions on $\R^d$
that can be used explicitly to construct a large number
of functions of positive type that are central to our analysis.
We will also use this opportunity to
introduce another vector space $\cup_{p>1}\mathcal{G}_p(\R^d)$ of functions
that will play a prominent role in later sections (though not in this one).

\begin{definition}\label{def:F_p}
	Choose and fix a real number $p>1$, and
	define $\mathcal{F}_p(\R^d)$ to be the collection of all $h \in L^p_{\text{\it loc}}(\R^d)$
	that satisfy
	\begin{equation}\label{cond:th:main:intro}
		\int_0^1 s^{d-1} \left( \| h \|_{L^p(\mathbb{B}_s)}
		 \| h \|_{L^q(\mathbb{B}_s^c)} + \| h \|_{L^2(\mathbb{B}_s^c)}^2\right)\d s<\infty,
	\end{equation}
	where $q:=p/(p-1)$.
	We also define $\mathcal{G}_p(\R^d)$ to be the collection of every
	function $h\in L^p_{\text{\it loc}}(\R^d)$ that satisfies \eqref{cond:omega}.
\end{definition}

In this section we study some of the basic properties of the elements of
the vector spaces
$\cup_{p>1}\mathcal{F}_p(\R^d)$ and $\cup_{p>1}\mathcal{G}_p(\R^d)$.
It might help to add that, notationally speaking, the functions $h$ in $\cup_{p>1}
\mathcal{G}_p(\R^d)$ and $\cup_{p>1}\mathcal{F}_p(\R^d)$ will be potential
candidates for the function $h$ in \eqref{f=h*h}, which are then used to form
the spatial correlation function $f$ in \eqref{SHE}. Thus, the notation should aid
the reading, and not hinder it.

\begin{lemma}\label{lem:F_p}
	The following are valid for every $p>1$, where $q:=p/(p-1)$:
	\begin{compactenum}
		\item $\mathcal{G}_p(\R^d)\subseteq
			\mathcal{F}_p(\R^d)\subseteq
			L^1_{\text{\it loc}}(\R^d)$ for all $d\ge1$, and
			$\mathcal{G}_p(\R)=\mathcal{F}_p(\R)$.
		\item $\|h\|_{L^p(\mathbb{B}_r)}$,
			$\|h\|_{L^q(\mathbb{B}_r^c)}$, and
			$\|h\|_{L^2(\mathbb{B}_r^c)}$ are  finite for every
			$h\in\mathcal{F}_p(\R^d)$ and $r>0$.
		\item If $h\in\mathcal{F}_p(\R^d)$, then
			\begin{equation}\label{eq:F_p}
				\int_0^r s^{d-1} \left( \| h \|_{L^p(\mathbb{B}_s)}
				 \| h \|_{L^q(\mathbb{B}_s^c)} +
				 \| h \|_{L^2(\mathbb{B}_s^c)}^2\right)\d s<\infty
				 \qquad\text{for every $r>0$.}
			\end{equation}
		\item If $h\in\mathcal{G}_p(\R^d)$, then
			\begin{equation}\label{eq:G_p}
				\int_0^r \left( \| h \|_{L^p(\mathbb{B}_s)}
				 \| h \|_{L^q(\mathbb{B}_s^c)} +
				 \| h \|_{L^2(\mathbb{B}_s^c)}^2\right)\omega_d(s)\,\d s<\infty
				 \qquad\text{for every $r>0$.}
			\end{equation}
	\end{compactenum}
\end{lemma}

\begin{proof}
	We have $\mathcal{G}_p(\R^d)\subset\mathcal{F}_p(\R^d)$ for all $d\ge2$
	and $\mathcal{G}_p(\R)=\mathcal{F}_p(\R)$ because
	of \eqref{eq:omega}; and the local integrability of $h\in\mathcal{F}_p(\R^d)$
	is a consequence of H\"older's inequality. This proves part 1.
	We concentrate on the remaining assertions of the lemma.
	
	First, let us note that if $p>1$ and $h \in\mathcal{F}_p(\R^d)$, then $h$ is locally in $L^p(\R^d)$
	and hence $\| h \|_{L^p(\mathbb{B}_r)}$ is finite for every $r>0$. In particular,
	\begin{equation}\label{3:phi}
		\| h \|_{L^q(\mathbb{B}_r^c)} + \| h \|_{L^2(\mathbb{B}_r^c)}<\infty,
	\end{equation}
	for almost every $r\in[0\,,1]$. Since both of the norms in \eqref{3:phi}
	are monotonically-decreasing functions of $r$,
	it follows that in fact \eqref{3:phi} holds for every $r>0$. This proves  part 2 of
	the lemma.
	
	Next, suppose $r>1$ and observe that
	\[
		\int_1^rs^{d-1} \left( \| h \|_{L^p(\mathbb{B}_s)}
		 \| h \|_{L^q(\mathbb{B}_s^c)} +
		 \| h \|_{L^2(\mathbb{B}_s^c)}^2\right)\d s
		 \le \left(\|h\|_{L^p(\mathbb{B}_r)}\|h\|_{L^q(\mathbb{B}_1^c)}
		 + \|h\|_{L^2(\mathbb{B}_1^c)}^2\right)\left(\frac{r^{d-1}-1}{d}\right)
	\]
	is finite. This and the definition of the vector space $\mathcal{F}_p(\R^d)$
	together imply that \eqref{eq:F_p} holds; \eqref{eq:G_p} is proved similarly.
\end{proof}

It follows from local integrability that the Fourier transform of every
function $ h \in\mathcal{F}_p(\R^d)$
($p>1$) is a well-defined distribution. In particular,  both $f= h *\tilde{ h }$ and
$| h |*|\tilde{ h }|$ are also well-defined distributions.
Of course, all such distributions are positive definite as well. The following shows that both
$h*\tilde{h}$ and $|h|*\tilde{|h|}$ are in fact fairly
nice positive-definite functions from $\R^d$ to the extended real numbers $\R\cup\{\infty\}$.
	
\begin{lemma}\label{lem:PD}
	If $ h \in\mathcal{F}_p(\R^d)$ for some $p>1$, then
	$ h *\tilde{ h }$ and $| h |*|\tilde{ h }|$ are functions of positive type. Moreoever,
	for every $r>0$,
	\begin{equation}\label{eq:PD}
		\sup_{\|x\|>2r} \left| \left(  h * \tilde{ h }\right)(x)\right| \le
		\sup_{\|x\|>2r} \left( | h |*|\tilde{ h }|\right)(x) \le
		2 \| h \|_{L^p(\mathbb{B}_r)}
		 \| h \|_{L^q(\mathbb{B}_r^c)} + \| h \|_{L^2(\mathbb{B}_r^c)}^2,
	\end{equation}
	and
	\begin{equation}\label{eq:PD:int}
		\int_{\mathbb{B}_r}\left( | h |*|\tilde{ h }|\right)(x)\,\d x
		\lesssim\int_0^{2r} s^{d-1}\left( \| h \|_{L^p(\mathbb{B}_s)}
		 \| h \|_{L^q(\mathbb{B}_s^c)} +
		 \| h \|_{L^2(\mathbb{B}_s^c)}^2\right)\d s,
	\end{equation}
	where the implied constant depends only on $d$.
\end{lemma}

\begin{proof}
	The argument  hinges loosely on old ideas that are
	motivated by the literature on potential theory of L\'evy processes; see in particular
	Hawkes \cite{Hawkes,Hawkes1984}.
	
	First of all, consider the case that $h$ is, in addition, non negative. In that case,
	$(h*\tilde{h})(x)$ is a well-defined Lebesgue integral for every $x\in\R^d$,
	though it might (or might not) diverge.
	We will show, among other things, that $(h*\tilde{h})(x)$ cannot diverge
	unless possibly when $x=0$. From here on,
	let us choose and fix some $r>0$ and $x\in\R^d$ such that $\|x\|>2r$.
	
	On one hand, if $y\in \mathbb{B}_r$ then certainly $\|x-y\|>r$, whence
	\[
		\int_{\mathbb{B}_r} | h(y) h(y-x) | \,\d y \le
		\| h\|_{L^p(\mathbb{B}_r)}\| h\|_{L^q(\mathbb{B}_r^c)},
	\]
	by H\"older's inequality.
	On the other hand, H\"older's inequality ensures that for every $z\in\R^d$,
	\begin{align*}
		\int_{\mathbb{B}_r^c} | h(y) h (y-z) | \,\d y &\le \int_{\substack{%
			\|y\|>r\\\|z-y\|<r}}| h (y) h (y-z) |\,\d y +
			\int_{\substack{%
			\|y\|>r\\\|z-y\|>r}}| h (y) h (y-z)|\,\d y\\
		&\le \| h\|_{L^p(\mathbb{B}_r)}\| h\|_{L^q(\mathbb{B}_r^c)} +
			\| h\|_{L^2(\mathbb{B}_r^c)}^2.
	\end{align*}
	If $h\in\mathcal{F}_p(\R^d)$ then certainly
	$|h|\in\mathcal{F}_p(\R^d)$ also.
	Combine the above bounds to obtain \eqref{eq:PD}.
	
	Now that we have established \eqref{eq:PD}, we obtain \eqref{eq:PD:int} by
	merely observing that
	\[
		\int_{\mathbb{B}_r} ( |h| *|\tilde{ h }|)(x)\,\d x
		\le \int_{\mathbb{B}_r}\Phi(\|x\|/2)\,\d x \propto
		\int_0^{2r}\Phi(s) s^{d-1}\,\d s,
	\]
	where $\Phi(t) := \sup_{\|x\|>2t}(|h| * |\tilde{ h }|)(x)$ for every $t>0$.
	Apply the already-proved part of the lemma,
	together with Lemma \ref{lem:F_p}, in order
	to see that $|h|*|\tilde{h}|\in L^1_{\text{\it loc}}(\R^d)$.
	
	Finally, we observe from the same argument that,
	whenever $h_1,h_2\in\mathcal{F}_p(\R^d)$,
	\[
		\sup_{\|x\|>2r}\left( |h_1|*|\tilde{h}_2|\right)(x)
		\le \|h_1\|_{L^p(\mathbb{B}_r)}\|h_2\|_{L^q(\mathbb{B}_r^c)}
		+\|h_2\|_{L^p(\mathbb{B}_r)}\|h_1\|_{L^q(\mathbb{B}_r^c)}
		+\|h_1\|_{L^2(\mathbb{B}_r^c)}\|h_2\|_{L^2(\mathbb{B}_r^c)}.
	\]
	Choose and fix an approximation
	to the identity $\{\varphi_\varepsilon\}_{\varepsilon>0}$ such that
	$\varphi_\varepsilon\in C^\infty_c(\R^d)$ for every
	$\varepsilon>0$.
	We may apply the preceding displayed inequality, once with $(h_1\,,h_2)=(h\,,h -(\varphi_\varepsilon*h))$
	and once with $(h_1\,,h_2)=(|h|\,,|h| -(\varphi_\varepsilon*|h|))$, in order to see that
	as $\varepsilon\downarrow0$,
	$(\varphi_\varepsilon* |h|*|\tilde{h}|)(x)
	\to (|h|*|\tilde{h}|)(x)$ and	
	$(\varphi_\varepsilon* h*\tilde{h})(x)
	\to (h*\tilde{h})(x)$, both valid uniformly for all $x\in\R^d$ that satisfy $\|x\|>2r$. This
	uses only the classical fact that
	\[
		\lim_{\varepsilon\downarrow0}
		\left( \left\| g - (\varphi_\varepsilon*g) \right\|_{L^p(\mathbb{B}_r)}
		+ \left\| g - (\varphi_\varepsilon*g) \right\|_{L^q(\mathbb{B}_r^c)}
		+ \left\| g - (\varphi_\varepsilon*g) \right\|_{L^2(\mathbb{B}_r^c)}
		\right)=0,
	\]
	for either $g=h$ or $g=|h|$ (see Stein \cite{Stein}), and readily implies the uniform
	continuity and boundedness of
	$h*\tilde{h}$ and $|h|*|\tilde{h}|$ off $\mathbb{B}_r$ for arbitrary $r>0$.
	This concludes the proof
	of the lemma.
\end{proof}

\section{On Condition (\ref{Dalang})}\label{sec:Dalang}

As was mentioned in the Introduction, it was shown by Dalang \cite{Dalang1999}
that when $f$ is tempered and
non negative (the latter being the more important condition),
Condition \eqref{Dalang} is an optimal sufficient condition
for the existence of a unique random-field solution to the SPDE \eqref{SHE}.
In this section, we say a few words about Dalang's Condition \eqref{Dalang}.

First recall that the vector space $H_{-1}(\R^d)$ denotes
the completion of all rapidly-decreasing, real-valued functions
on $\R^d$ in the norm
\[
	\| h \|_{H_{-1}(\R^d)} :=
	\left( \int_{\R^d} \frac{|\widehat{ h }(x)|^2}{1+\|x\|^2}\,\d x\right)^{1/2}.
\]
It follows immediately that $H_{-1}(\R^d)$ is Hilbertian, once endowed with the above norm and
the associated inner product,
\[
	\langle \psi_1\,, \psi_2\rangle_{H_{-1}(\R^d)} :=
	\int_{\R^d}\frac{\widehat{\psi}_1(x)\overline{\widehat{\psi}_2(x)}}{1+\|x\|^2}\,\d x.
\]
This explains why the two conditions in
\eqref{Dalang} are equivalent.

Next, let us define $\bm{v}_\lambda$ to be the \emph{$\lambda$-potential density}
of the heat semigroup on $\R^d$
for every $\lambda>0$. That is,
\begin{equation}\label{HRD}
	\bm{v}_\lambda(x) = \int_0^\infty\e^{-\lambda t} \bm{p}_t(x)\,\d t
	\qquad\text{for all $x\in\R^d$},
\end{equation}
where $\bm{p}$ denotes the heat kernel, defined as
\begin{equation}\label{p}
	\bm{p}_t(x) := \frac{1}{(2\pi t)^{d/2}}\exp\left(-\frac{\|x\|^2}{2t}\right)\qquad
	\text{for all $t>0$ and $x\in\R^d$}.
\end{equation}

Note that $\lambda \bm{v}_\lambda$ is  a probability density function on $\R^d$
for every $\lambda>0$.

A general theorem of Foondun and Khoshnevisan \cite{FoondunKhoshnevisan2013} implies that, when
$h\ge0$ (and hence $f\ge0$), Dalang's condition
\eqref{Dalang} holds if and only if \footnote{In general, the proof of \eqref{FK}
	requires some effort. But, for example when $h\in L^1(\R^d)\cap L^2(\R^d)$,
	Young's inequality yields $f\in \cap_{\nu\in[1,\infty]} L^\nu(\R^d)$
	and hence \eqref{FK} is a direct consequence of Parseval's identity and the elementary facts that:
	(i) The Fourier transform of $\bm{v}_\lambda$ is
	$\widehat{\bm{v}}_\lambda(z) :=\int_{\R^d}\exp\{ix\cdot z\}\bm{v}_\lambda(x)\,\d x
	= 2[2\lambda + \|z\|^2]^{-1}$ for all $z\in\R^d$; and
	(ii) $\widehat{f}(z) := \int_{\R^d}\exp\{ix\cdot z\}f(x)\,\d x
		=|\hat{h}(z)|^2$ for all $z\in\R^d$.}
\begin{equation}\label{FK}
	\int_{\R^d}\bm{v}_\lambda(x)f(x)\,\d x<\infty
	\quad\text{for one, hence all, $\lambda>0$}.
\end{equation}
An earlier result, applicable in the present context, can be found in Peszat \cite[Theorem 0.1]{Peszat2002}.

While \eqref{Dalang} and \eqref{FK} are equivalent formulations
of the same condition, each formulation has its technical advantages:
On one hand, it is clear from Condition \eqref{Dalang} that if $h$ is a nonnegative element of
$H_{-1}(\R^d)$ then $\varphi*h$ is also a nonnegative element of
$H_{-1}(\R^d)$ for every probability density $\varphi$ on $\R^d$. Thus, we see from \eqref{Dalang}
that $H_{-1}(\R^d)$ is closed under ``smoothing.''

On the other hand, the phrasing of Condition \eqref{FK} obviates the assertion that $H_{-1}(\R^d)$
is closed under ``miniorization'': If $0\le h_1\le h_2$ are measurable
and $h_2\in H_{-1}(\R^d)$, then $h_1$ is in $H_{-1}(\R^d)$ also. This 
minorization property will play a keyrole in the proof of spatial ergodicity in Theorem
\ref{th:main:intro}.

Let us note also that if $h\ge0$ and $h\in\mathcal{F}_p(\R^d)$
for some $p>1$, then $f$ is bounded uniformly on $\mathbb{B}_r^c$ for all $r>0$.
Because in addition $\bm{v}_\lambda$ is integrable,
it follows from \eqref{FK} that, in the present setting wherein
$h\ge0$ and $h\in\cup_{p>1}\mathcal{F}_p(\R^d)$,
the harmonic-analytic condition \eqref{Dalang}---equivalently
the potential-theoretic condition \eqref{FK}---%
is equivalent to the following local version of \eqref{FK}:
\begin{equation}\label{FK:loc}
	\int_{\mathbb{B}_1}\bm{v}_\lambda(x)f(x)\,\d x<\infty
	\qquad\text{for one, hence all, $\lambda>0$.}
\end{equation}
Next, we re-interpret \eqref{FK:loc}:
It is well known, and easy to verify directly (see, for example, Khoshnevisan
\cite[Section 3.1, Chapter 10]{KhMPP}), that
\begin{equation}\label{pot:bound}
	\bm{v}_\lambda(x)\asymp \|x\|^{-d+1}\omega_d(\|x\|)
	\qquad\text{uniformly for all $x\in\mathbb{B}_1$},
\end{equation}
where $\omega_d$ was defined in \eqref{eq:omega}.
Thus, when $h\ge0$ and $h\in\cup_{p>1}\mathcal{F}_p(\R^d)$,
\begin{equation}\label{h:kappa}
	h\in H_{-1}(\R^d)\quad\text{iff}\quad
	\int_{\mathbb{B}_1} \|x\|^{-d+1}\omega_d(\|x\|)f(x)\,\d x<\infty.
\end{equation}

Now consider the general case where $h\in\cup_{p>1}\mathcal{F}_p(\R^d)$
is possibly signed. Define
\begin{equation}\label{bar(f)}
	\bar{f}(r) := \sup_{\|x\|>r}\left( |h|*|\tilde{h}|\right)(x)
	\qquad\text{for every $r>0$}.
\end{equation}
Since $|f(x)|\le\bar{f}(\|x\|)$ for all $x\in\R^d$,
we can apply \eqref{h:kappa} with $(h\,,f)$ replaced with $(|h|\,,|h|*|\tilde{h}|)$
in order to see that
\[
	\text{if}\quad\int_0^1\bar{f}(r)\omega_d(r)\,\d r<\infty,
	\quad\text{then}\quad |h|\in H_{-1}(\R^d).
\]
If $f\ge0$ and $x\mapsto f(x)$ is a radial function on $\R^d$ that decreases with $\|x\|$,
then $\bar{f}(\|x\|)= f(x)$, and the above sufficient condition for $|h|=h$
to be in $H_{-1}(\R^d)$  appears
earlier in the literature, in the context of well-posedness for SPDEs. 
See Dalang and Frangos \cite{DalangFrangos}, Karczewska
and Zabczyk\cite{KZ}, Peszat \cite{Peszat2002},
and Peszat and Zabczyk \cite{PeszatZabczyk2000}. 
Closely-related results can be found in Cardon-Weber and Millet
\cite{CardonWeberMillet}, Dalang \cite{Dalang1999}, 
Foondun and Khoshnevisan \cite{FoondunKhoshnevisan2013},
and Millet and Sanz-Sol\'e \cite{MilletSanz}.

Recall the vector space $\cup_{p>1}\mathcal{G}_p(\R^d)$ (Definition
\ref{def:F_p}) and the inequalities of Lemma \ref{lem:PD}  in order to deduce the following.

\begin{lemma}\label{lem:Dalang:Lp}
	If $h\in\mathcal{G}_p(\R^d)$ for some $p>1$, then $|h|\in H_{-1}(\R^d)$.
	In particular, $h\in\cup_{p>1}\mathcal{G}_p(\R^d)$ implies that
	$\int_{\R^d}\bm{v}_\lambda(x)|f(x)|\,\d x<\infty$
	for some, hence all, $\lambda>0$.
\end{lemma}

In light of Theorem 1.2 of Foondun and Khoshnevisan \cite{FoondunKhoshnevisan2013},
Lemma \ref{lem:Dalang:Lp} implies a precise version of the somewhat
subtle assertion that sufficient  integrability of $h$ ensures
good decay at infinity of the Fourier transform of $|h|$.

\section{Stochastic convolutions}\label{sec:SC}
If $\Phi=\{\Phi(t\,,x)\}_{t\ge0,x\in\R^d}$ is a space-time random field, then
for all real numbers $\beta>0$ and $k\ge1$, we may define
\begin{equation}\label{N}
	\mathcal{N}_{\beta,k}(\Phi) := \sup_{t\ge0}\sup_{x\in\R^d}\e^{-\beta t}\|\Phi(t\,,x)\|_k.
\end{equation}
It is clear that $\Phi\mapsto\mathcal{N}_{\beta,k}(\Phi)$ defines a norm for
every choice of $\beta>0$ and $k\ge1$.
These norms were first introduced in \cite{FoondunKhoshnevisan2009}; see also
\cite{ConusKh2012}.
Corresponding to every $\mathcal{N}_{\beta,k}$, define $\mathbb{W}_{\beta, k}$ to
be the collection of all predictable random fields
$\Phi$ such that $\mathcal{N}_{\beta,k}(\Phi)<\infty$. We may think of elements of
$\mathbb{W}_{\beta,2}$ as \emph{Walsh-integrable random fields with Lyapunov exponent
$\le\beta$}. It is easy to see that each
$(\mathbb{W}_{\beta,2}\,,\mathcal{N}_{\beta,k})$ is a Banach space.

Suppose that the underlying probability space $(\Omega\,,\mathcal{F},\P)$ is large enough
to carry a space-time white noise $\xi$ (if not then enlarge it in the usual way).
Using that noise, we may formally
define, for every fixed measurable function $h:\R^d\to\R$,  a new noise $\eta^{(h)}$ as follows:
\begin{equation}\label{eta^h}
	\eta^{( h )}(\d s\,\d x) := \int_{\R^d} h (x-y)\,\xi(\d s\,\d y)\,\d x.
\end{equation}
Somewhat more precisely, if $H$ is a predictable random field such that
\[
	\E\int_0^t\d s\int_{\R^d}\d y\
	\left|  \left( H(s)*\tilde{ h }\right)(y)\right|^2<
	\infty\qquad\text{for every $t>0$},
\]
then Walsh's theory of stochastc integration ensures that the Walsh stochastic integral
\[
	\int_{(0,t)\times\R^d} H(s\,,x)\,\eta^{( h )}(\d s\,\d x)
	:= \int_{(0,t)\times\R^d} \left( H(s)*\tilde{ h }\right)(y)\,\xi(\d s\,\d y)
\]
is well-defined for every $t\ge0$, and in fact defines a continuous, mean-zero,
$L^2(\P)$ martingale indexed by $t\ge0$. Moreover, the variance of this martingale
at time $t>0$ is
\begin{align}\label{SID}
	\E\left(\left| \int_{(0,t)\times\R^d} H(s\,,x)\,\eta^{( h )}(\d s\,\d x)\right|^2\right)
		&=\E\int_0^t\d s\int_{\R^d}\d y\
		\left|  \left( H(s)*\tilde{ h }\right)(y)\right|^2\\\notag
	&=\int_0^t\d s\int_{\R^d}\d y\int_{\R^d}\d z\
		\E\left[H(s\,,y)H(s\,,z)\right]f(y-z),
\end{align}
provided for example that the preceding integral is absolutely convergent.
(As it is case elsewhere in this paper, $f$ is defined in terms of $h$ via \eqref{f=h*h}.)
It is easy to see from this that $\eta^{(h)}$ is a particular construction of the noise
$\eta$ of the Introduction (see also Conus et al \cite{CJKS2013}), but has the advantage that
it provides a coupling $h\mapsto \eta^{(h)}$ that works simultaneously for many different
choices of $h$, whence spatial correlation functions $f$.

The preceding stochastic integration (see \eqref{SID}) frequently
allows for the integeration of a large family of
predictable random fields $H$. The following simple result
highlights a large subclass of such  random fields when $h\in\cup_{p>1}\mathcal{F}_p(\R^d)$.

\begin{lemma}\label{lem:convergent}
	Suppose $ h \in\mathcal{F}_p(\R^d)$ for some $p>1$, and $H$ is
	a predictable process for which there exists a real number $r>0$ such that
	\begin{equation}\label{cond:H}
		\sup_{s\in[0,T]}\sup_{y\in\R^d}\E\left(|H(s\,,y)|^2\right)<\infty
		\quad\text{and}\quad
		\E\left(|H(t\,,x)|^2\right)=0\quad
		\text{for every $t>0$ and $x\in \mathbb{B}_r^c$}.
	\end{equation}
	Then, the final integral in \eqref{SID}
	is absolutely convergent and hence \eqref{SID} is valid  for every $t>0$.
\end{lemma}

In particular, we may consider arbitrary non-random functions $H:\R_+\times\R^d\to\R$
of compact support in order to learn from Lemma \ref{lem:convergent}
and the second identity in \eqref{SID} that when $h\in\cup_{p>1}\mathcal{F}_p(\R^d)$,

\begin{proof}[Proof of Lemma \ref{lem:convergent}]
	Choose and fix an arbitrary $t>0$. In accord with our earlier
	remarks, and thanks to the Cauchy--Schwarz inequality,
	it suffices to prove that
	\[
		J := \int_0^t\d s\int_{\mathbb{B}_r}\d y\int_{\mathbb{B}_r}\d z\ \|H(s\,,y)\|_2\|H(s\,,z)\|_2
		\left| f(y-z)\right|
		<\infty\qquad\text{for every $t>0$}.
	\]
	But the triangle inequality readily yields
	\[
		J\le |\mathbb{B}_r|
		\left(\int_0^t\sup_{y\in\R^d}\|H(s\,,y)\|_2^2\,\d s\right) \left(
		\int_{\mathbb{B}_{2r}} \left(| h |*|\tilde{ h }|\right)(w)\d w\right),
	\]
	which is finite thanks to \eqref{cond:H} and Lemma \ref{lem:PD}; see in particular
	\eqref{eq:PD:int}.
\end{proof}

The second portion of \eqref{cond:H} involves a compact-support condition which can sometimes
be reduced to a decay-type condition. We exemplify that next for a specific
family of the form $H(s\,,y) = \bm{p}_{t-s}(x-y)Z(s\,,y)$, where $t>s$ and $x\in\R^d$
are fixed and $\bm{p}$ denotes the heat kernel
[see \eqref{p}]. 
With this choice, the following ``stochastic convolution'' is a well-defined random field
provided that it is indeed defined properly as a Walsh integral for every $t>0$ and $x\in\R^d$:
\begin{equation}\label{SC}
	\left( \bm{p}\circledast Z\eta^{( h )}\right)(t\,,x) := \int_{(0,t)\times\R^d}
	\bm{p}_{t-s}(x-y) Z(s\,,y)\,\eta^{( h )}(\d s\,\d y).
\end{equation}

For every $k\ge2$, let $z_k^k$ denote the optimal constant of the $L^k(\P)$-form of
the Burkholder--Davis--Gundy inequality \cite{Burkholder,BDG,BG}; 
that is, for every continuous
$L^2(\P)$-martingale $\{M_t\}_{t\ge0}$,
and all real numbers $k\ge2$ and $t\ge0$,
\[
	\E\left( |M_t|^k\right) \le z_k^k\E\left( \langle M\rangle_t^{k/2}\right).
\]
Then,
\begin{equation}\label{z_k}
	z_2=1\qquad\text{and }\qquad
	z_k\le 2\sqrt{k}\qquad\text{for every $k>2$}.
\end{equation}
The first assertion is the basis of It\^o's stochastic calculus, and the second is due
to Carlen and Kree \cite{CarlenKree1991}, who also proved that $\lim_{k\to\infty} (z_k/\sqrt k) = 2$.
The exact value of $z_k$ is computed in the celebrated paper of Davis \cite{Davis1976}.

The following provides a natural condition for the stochastic convolution to be a well-defined
random field, the stochastic integral being defined in the sense of Walsh \cite{Walsh}, and extends
Propositiuon 6.1
of Conus et al \cite{CJKS2013} to the case that $f$ is possibly signed.
It might help to recall that $\bm{v}_\beta$ denotes the $\beta$-potential
kernel [see \eqref{HRD}].

\begin{lemma}[A stochastic Young inequality]\label{lem:Young}
	Suppose that $Z\in\mathbb{W}_{\beta,k}$ for some $\beta>0$, $k\ge2$,
	and that $ h \in\mathcal{G}_p(\R^d)$ for some $p>1$. Then,
	the stochastic convolution in \eqref{SC} is a well-defined Walsh
	integral,
	\[
		\mathcal{N}_{\beta,k}\left(\bm{p}\circledast Z\eta^{( h )}\right)
		\le z_k \mathcal{N}_{\beta,k}(Z) \cdot\sqrt{\frac12\int_{\R^d} \bm{v}_\beta(x)
		|f(x)|\,\d x},
	\]
	and the integral under the square root is finite.
\end{lemma}

\begin{proof}
	The integral under the square root is finite thanks to Lemma \ref{lem:Dalang:Lp}.
	We proceed to prove the remainder of the lemma.
	
	According to the theory of Walsh \cite{Walsh}, the random field
	$\bm{p}\circledast Z\eta^{( h )}$ is
	well defined whenever $\mathcal{Q}_2(t\,,x)<\infty$ where
	\[
		\mathcal{Q}_k(t\,,x):=
		\int_0^t\d s\int_{\R^d}\d y\int_{\R^d}\d z\ \bm{p}_{t-s}(x-y)
		\bm{p}_{t-s}(x-z)
		\|Z(s\,,y)\|_k\|Z(s\,,z)\|_k| f(y-z) |
	\]
	for every $t>0$ and $x\in\R^d$.
	Moreover (see also \eqref{SID}), in that case, the Burkholder--Davis--Gundy
	inequality yields
	\begin{align*}
		&\E\left(\left|  \left(
			\bm{p}\circledast Z\eta^{( h )}\right)(t\,,x)\right|^k\right)\\
		&\le z_k^k\E\left(\left|
			\int_0^t\d s\int_{\R^d}\d y
			\int_{\R^d}\d z\ \bm{p}_{t-s}(x-y)\bm{p}_{t-s}(x-z)
			Z(s\,,y)Z(s\,,z) f(y-z)\right|^{k/2}\right)\\
		&\le z_k^k\left[ \int_0^t\d s\int_{\R^d}\d y
			\int_{\R^d}\d z\ \bm{p}_{t-s}(x-y)\bm{p}_{t-s}(x-z)
			\|Z(s\,,y)Z(s\,,z)\|_{k/2} |f(y-z)|\right]^{k/2}\\
		&\le z_k^k\left[ \mathcal{Q}_k(t\,,x) \right]^{k/2},
	\end{align*}
	the last line holding thanks to the Cauchy--Schwarz inequality.
	It remains to prove that $\mathcal{Q}_k(t\,,x)<\infty$ for all $t>0$ and $x\in\R^d$.
	
	Since $\|Z(s\,,y)\|_k\le \exp(\beta s)\mathcal{N}_{\beta,k}(Z)$ for all $s\ge0$ and $y\in\R^d$,
	it then follows that
	\begin{align*}
		\mathcal{Q}_k(t\,,x)&\le \left[\mathcal{N}_{\beta,k}(Z)\right]^2
			\int_0^t\e^{-2\beta s}\,\d s\int_{\R^d}\d y\int_{\R^d}\d z\
			\bm{p}_{t-s}(x-y)\bm{p}_{t-s}(x-z)
			\left| f(y-z)\right|\\
		&\le \e^{2\beta t} \left[\mathcal{N}_{\beta,k}(Z)\right]^2
			\int_0^t\e^{-2\beta r}\,\d r \int_{\R^d}\d w\ \bm{p}_{2r}(w)
			|f(w)|,
	\end{align*}
	after two change of variables $[w=y-z,\ r=t-s]$, and
	thanks to the Chapman-Kolmogorov (semigroup)
	property of the heat kernel $\bm{p}$. Since
	\[
		\int_0^t \exp(-2\beta r)\bm{p}_{2r}(w)\,\d r\le\int_0^\infty
		\exp(-2\beta r)\bm{p}_{2r}(w)\,\d r=\tfrac12 \bm{v}_\beta(w),
	\]
	for every $w\in\R^d$ and $\beta>0$, this proves that
	\[
		\e^{-2\beta t}\mathcal{Q}_k(t\,,x)
		\le \tfrac12\left[\mathcal{N}_{\beta,k}(Z)\right]^2
		\int_{\R^d} \bm{v}_\beta(w)
		|f(w)|.
	\]
	This inequality completes the proof of the lemma upon taking square roots,
	as the right-hand side of the preceding inequality is independent of $(t\,,x)$.
\end{proof}

The following localization result has its roots in the earlier work of Conus et al
\cite{CJKS2013}. The advantage of
the present version is that it is based on the real-analytic study of positive-definite
functions, as exemplified in this paper, and not on Fourier-analytic methods of
Ref.\ \cite{CJKS2013} which
work only under additional technical conditions.

\begin{lemma}\label{lem:localize:noise}
	Suppose $h\in\mathcal{G}_p(\R^d)$ for some $p>1$ and
	$Z\in\mathbb{W}_{\beta,k}$
	for some $\beta>0$ and $k\ge2$. Define
	\begin{equation}\label{h:_r}
		h_r(x) := h(x) \bm{1}_{\mathbb{B}_r}(x)\qquad\text{for
		every $r>0$ and $x\in\R^d$}.
	\end{equation}
	Then, $h_r\in\mathcal{G}_p(\R^d)$ for all $r>0$, and
	\[
		\mathcal{N}_{\beta,k}\left(\bm{p}\circledast Z\eta^{( h )} -
		\bm{p}\circledast Z\eta^{( h_r)}\right) \le
		\frac{z_k \| h \|_{L^2(\mathbb{B}_r^c)}}{\sqrt{2\beta}}\,
		\mathcal{N}_{\beta,k}(Z)
		\qquad\text{for all $r>0$}.
	\]
\end{lemma}

\begin{proof}
	Because $|h_r(x)|\le|h(x)|$ for all $x\in\R^d$ and $h\in\mathcal{G}_p(\R^d)$, it follows immediately that
	$h_r\in\mathcal{G}_p(\R^d)$ for every $r>0$. Thus,
	we need only concentrate on the norm inequality of the lemma.
	
	Define $\Phi_r :=  h  -  h_r =  h \bm{1}_{\mathbb{B}_r^c}$, and observe that
	\begin{align*}
		\mathcal{N}_{\beta,k}\left(\bm{p}\circledast Z\eta^{( h )} -
		\bm{p}\circledast Z\eta^{( h_r)}\right)
		&= \mathcal{N}_{\beta,k}\left(\bm{p}\circledast Z\eta^{(\Phi_r)} \right)\\
		&\le z_k\mathcal{N}_{\beta,k}(Z)\;\sqrt{\frac12\int_{\R^d}
		\bm{v}_\beta(x) \left|\left( \Phi_r* \widetilde{\Phi_r}
		\right)(x)\right|\d x}\:,
	\end{align*}
	thanks to Lemma \ref{lem:Young}.
	Young's inequality implies that
	\[
		\left|\left( \Phi_r*\widetilde{\Phi_r}\right)(x)\right|
		\le \| h \|_{L^2(\mathbb{B}_r^c)}^2
		\qquad\text{for every $x\in\R^d$},
	\]
	whence follows the result since $\beta\bm{v}_\beta$ is a probability density
	function on $\R^d$.
\end{proof}

\section{Well posedness}\label{sec:wellposedness}

Before we study the spatial ergodicity of the solution to \eqref{SHE} we address matters
of well posedness. As was mentioned earlier, well-posedness follows from the more general
theory of Dalang \cite{Dalang1999} when $h\ge0$, for example.
Here we say a few things about general
well posedness when $h$ is signed. This undertaking does require some new ideas, but most
of those new ideas have already been developed in the earlier sections, particularly as regards
the space $\cup_{p>1}\mathcal{G}_p(\R^d)$, which now plays a prominent role.

Recall that $\lambda$-potential $\bm{v}_\lambda$ from \eqref{HRD}.
Choose and fix a function $h\in\cup_{p>1}\mathcal{G}_p(\R^d)$
and recall from Lemma \ref{lem:Dalang:Lp}
that
\[
	\int_{\R^d}\bm{v}_\lambda(x)|f(x)|\,\d x\le
	\int_{\R^d}\bm{v}_\lambda(x)\left( |h|*|\tilde{h}|\right)(x)\,\d x<\infty,
\]
for one, hence all,
$\lambda>0$. As a consequence, we find that the following is a well-defined,
$(0\,,\infty)$-valued function on $(0\,,\infty)$:
\begin{equation}\label{Lambda}
	\Lambda_h(\delta) := \inf\left\{ \lambda>0:\ \int_{\R^d} \bm{v}_\lambda(x)
	\left( |h|*|\tilde{h}|\right)(x)\,\d x<
	\delta\right\}\qquad\text{for all $\delta>0$},
\end{equation}
where $\inf\varnothing:=\infty$.

\begin{theorem}\label{th:exist}
	Suppose $h\in\mathcal{G}_p(\R^d)$ for some $p>1$.
	Then, the SPDE \eqref{SHE}, subject to non-random initial data $u(0)=u_0\in L^\infty(\R^d)$
	and non degeneracy condition $\text{\rm Lip}(\sigma)>0$,
	has a mild solution $u$ which is unique (upto a modification)
	subject to the additional condition that
	\[
		\sup_{x\in\R^d}\E\left(|u(t\,,x)|^k\right) \le
		\left[ \frac{\|u_0\|_{L^\infty(\R^d)}}{\varepsilon} +
		\frac{|\sigma(0)|}{\varepsilon\text{\rm Lip}(\sigma)}\right]^k
		\exp\left\{ kt\Lambda_h\left(
		\frac{2(1-\varepsilon)^2}{[z_k\text{\rm Lip}(\sigma)]^2}
		\right)\right\}
		\qquad\text{for all $t>0$},
	\]
	valid for every $\varepsilon\in(0\,,1)$ and $k\ge2$; see also \eqref{Lambda}.
	Finally, $(t\,,x)\mapsto u(t\,,x)$ is continuous in $L^k(\P)$ for very $k\ge2$,
	and hence Lebesgue measurable (upto evanescance).
\end{theorem}

\begin{remark}
	A ready by-product of Theorem \ref{th:exist} is that the $k$th moment
	Lyapunov exponent $\overline\lambda(k)$ of $u$ exists and satisfies
	\[
		\overline\lambda(k) := \limsup_{t\to\infty}\frac1t\sup_{x\in\R^d}
		\log\E\left(|u(t\,,x)|^k\right)
		\le k\Lambda_h\left( \frac{2}{[z_k
		\text{\rm Lip}(\sigma)]^2}\right) \begin{cases}
			= 2\Lambda_h\left( \frac{2}{[\text{\rm Lip}(\sigma)]^2}\right)
				&\text{if $k=2$};\\
			\le k\Lambda_h\left( \frac{1}{2k[
				\text{\rm Lip}(\sigma)]^2}\right)
				&\text{if $k>2$}.
		\end{cases}
	\]
	See \eqref{z_k}. More information on this topic can be found in \cite{KhCBMS}.
\end{remark}

\begin{remark}\label{rem:doob}
	Because of $L^k(\P)$-continuity, Doob's theory of separability becomes applicable (see
	Doob \cite{Doob})
	and implies, among other things, that $x\mapsto u(t\,,x)$ is Lebesgue measurable.
	This is of course directly relevant to the present discussion of spatial ergodicity.
\end{remark}

\begin{proof}[Outline of the proof of Theorem \ref{th:exist}]
	The proof follows a standard route. We therefore outline it,
	in part to document the verasity of the agrument, but
	mainly as a means of introducing objects that we will need later on.
	
	Let $u_0(t\,,x):=u_0(x)$ for all $t\ge0$ and $x\in\R^d$, and define
	iteratively
	\begin{align*}
		u_{n+1}(t\,,x) :=&\ \int_{\R^d}\bm{p}_t(y-x)u_0(y)\,\d y +
			\int_{(0,t)\times\R^d}\bm{p}_{t-s}(x-y)\sigma(u_n(s\,,y))
			\,\eta^{( h )}(\d s\,\d y)\\
		=&\ (\bm{p}_t*u_0)(t) + \left( \bm{p}\circledast
			\sigma(u_n)\eta^{( h )}\right)(t\,,x),
	\end{align*}
	for every integer $n\ge0$ and all real numbers $t\ge0$ and $x\in\R^d$.
	Since the first term is bounded uniformly by $\|u_0\|_{L^\infty(\R^d)}$,
	and since every $\mathcal{N}_{\beta,k}$ is a norm for every $\beta>0$
	and $k\ge1$,
	it follows that for all integers $n\ge0$, and all reals $\beta>0$ and $k\ge2$,
	\begin{equation}\label{N(u)}\begin{split}
		\mathcal{N}_{\beta,k}(u_{n+1}) &\le \|u_0\|_{L^\infty(\R^d)} +
			\mathcal{N}_{\beta,k}\left( \bm{p}\circledast
			\sigma(u_n)\eta^{( h )}\right)\\
		&\le\|u_0\|_{L^\infty(\R^d)} + z_k
			\mathcal{N}_{\beta,k}\left( \sigma(u_n)\right)\sqrt{\frac12\int_{\R^d}
			\bm{v}_\beta(x) |f(x)|\,\d x};
	\end{split}\end{equation}
	see Lemma \ref{lem:Young}.
	Because $|\sigma(z)|\le|\sigma(0)|+\text{\rm Lip}(\sigma)|z|$
	for all $z\in\R$, it follows that
	\[
		\mathcal{N}_{\beta,k}(u_{n+1}) \le\|u_0\|_{L^\infty(\R^d)} + z_k
		\left(|\sigma(0)|+\text{\rm Lip}(\sigma)\mathcal{N}_{\beta,k}
		( u_n)\right)
		\sqrt{\frac12\int_{\R^d}
		\bm{v}_\beta(x) \left( |h|*|\tilde{h}|\right)(x)\,\d x}.
	\]
	This is valid for every $\beta>0$ and $k\ge2$. Because
	\begin{equation}\label{E:Lambda-Beta}
		 \beta\ge \Lambda_h
		 \left(\frac{2(1-\varepsilon)^2}{[z_k\text{\rm Lip}(\sigma)]^2}\right)
		 \quad\text{iff}\quad
		 \int_{\R^d}\bm{v}_\beta(x) \left( |h|*|\tilde{h}|\right)(x)\,\d x
		 \le \frac{2(1-\varepsilon)^2}{[z_k\text{\rm Lip}(\sigma)]^2},
	\end{equation}
	it follows that
	\begin{equation}\label{eq:N(u_n+1)}\begin{split}
		\mathcal{N}_{\beta,k}(u_{n+1}) &\le\|u_0\|_{L^\infty(\R^d)} + z_k
			|\sigma(0)|\sqrt{\frac12\int_{\R^d}
			\bm{v}_\beta(x) \left( |h|*|\tilde{h}|\right)\,\d x} + (1-\varepsilon)
			\mathcal{N}_{\beta,k}(u_n)\\
		&\le \|u_0\|_{L^\infty(\R^d)} +
			\frac{|\sigma(0)|}{\text{\rm Lip}(\sigma)} +
			(1-\varepsilon)\mathcal{N}_{\beta,k}(u_n)\\
		&\le \|u_0\|_{L^\infty(\R^d)} +
			\frac{|\sigma(0)|}{\text{\rm Lip}(\sigma)} + (1-\varepsilon)\left[
			\|u_0\|_{L^\infty(\R^d)} +
			\frac{|\sigma(0)|}{\text{\rm Lip}(\sigma)}\right] + (1-\varepsilon)^2
			\mathcal{N}_{\beta,k}(u_{n-1})\\
		&\le\cdots\le \left[\|u_0\|_{L^\infty(\R^d)} +
			\frac{|\sigma(0)|}{\text{\rm Lip}(\sigma)}\right]\cdot
			\left[\sum_{j=0}^n(1-\varepsilon)^j +
			(1-\varepsilon)^{n+1}\|u_0\|_{L^\infty(\R^d)}\right]\\
		&\le \left[ \|u_0\|_{L^\infty(\R^d)} +
			\frac{|\sigma(0)|}{\text{\rm Lip}(\sigma)}\right]\cdot
			\left[\frac1\varepsilon +
			(1-\varepsilon)^{n+1}\|u_0\|_{L^\infty(\R^d)}
			\right],
	\end{split}\end{equation}
	after iteration. Similarly, one finds that
	\begin{equation}\label{u_n-u}\begin{split}
		\mathcal{N}_{\beta,k}(u_{n+1} - u_n) &\le
			\mathcal{N}_{\beta,k}\left( \bm{p}\circledast
			\left[ \sigma(u_n)-\sigma(u_{n-1})\right]\eta^{( h )}\right)\\
		&\le z_k\mathcal{N}_{\beta,k}\left( \sigma(u_n)-\sigma(u_{n-1})\right)
			\sqrt{\frac12\int_{\R^d}
			\bm{v}_\beta(x) \left( |h|*|\tilde{h}|\right)\,\d x}\\
		&\le z_k\text{\rm Lip}(\sigma)\mathcal{N}_{\beta,k}
			\left( u_n-u_{n-1}\right)\sqrt{\frac12\int_{\R^d}
			\bm{v}_\beta(x) \left( |h|*|\tilde{h}|\right)\,\d x}\\
		&\le(1-\varepsilon)\mathcal{N}_{\beta,k}\left( u_n-u_{n-1}\right),
	\end{split}\end{equation}
	provided still that
	$\beta \ge \Lambda_h(2(1-\varepsilon)^2/[z_k\text{\rm Lip}(\sigma)]^2)$.
	It follows immediately that $\{u_n\}_{n\ge0}$ is a Cauchy sequence in $\mathbb{W}_{\beta, k}$
	when $\beta\ge\Lambda_h(2(1-\varepsilon)^2/[z_k\text{\rm Lip}(\sigma)]^2)$. It also implies readily
	that  $u:=\lim_{n\to\infty}u_n$ is an element of $\mathbb{W}_{\beta, k}$, for the same range
	of $\beta$'s, and that $u$ solves \eqref{SHE}. This and Fatou's lemma together prove the asserted
	upper bound for $\E(|u(t\,,x)|^k)$ as well.
	
	The proof of uniqueness is also essentially standard: Suppose there existed
	$u,v\in\mathbb{W}_{\beta, k}$ for some
	$\beta \ge \Lambda_h(2(1-\varepsilon)^2/[z_k\text{\rm Lip}(\sigma)]^2)$
	both of which are mild solutions to \eqref{SHE}. Then, the same argument that led to
	\eqref{u_n-u} yields
	$\mathcal{N}_{\beta,k,T}(u-v)\le(1-\varepsilon)\mathcal{N}_{\beta,k,T}(u-v)$
	for all $\beta \ge \Lambda_h(2(1-\varepsilon)^2/[z_k\text{\rm Lip}(\sigma)]^2)$
	and $T>0$, where
	\[
		\mathcal{N}_{\beta,k,T}(\Phi) := \sup_{t\in[0,T]}\sup_{x\in\R^d}
		\e^{-\beta t}\|\Phi(t\,,x)\|_k;
	\]
	compare with \eqref{N}.
	In particular, it follows that there exists $\beta>0$ such that
	\[
		\mathcal{N}_{\beta,k,T}(u-v)=0 \qquad\text{for all $T>0$},
	\]
	and hence $u$ and $v$ are modifications of one another. We can unscramble the latter
	displayed statement in order to see that this yields the asserted bound for
	$\E(|u(t\,x)|^k)$. Similarly, one proves $L^k(\P)$ continuity, which completes our (somewhat
	abbreviated) proof of Theorem \ref{th:exist}.
\end{proof}

\section{Proof of stationarity}\label{sec:stat}
For every $\varphi\in C(\R_+\times\R^d)$ and $y\in\R^d$ define shift operators
$\{\theta_y\}_{y\in\R^d}$ as follows:
\[
	(\varphi\circ \theta_y)(t\,,x) = \varphi(t\,,x+y).
\]
Clearly, $\theta:=\{\theta_y\}_{y\in\R^d}$ is a group under composition. The
following is used tacitly in the literature many times without explicit proof of even mention
(see for example \cite{CJK2013}). It also improves the assertion, found observed by Dalang \cite{Dalang1999} that the 2-point correlation function of $x\mapsto u(t\,,x)$
is invariant under $\theta$. When $\sigma(z)\propto z$
the latter moment invariance (and more) can be deduced directly
from an explicit Feynman--Kac type moment formula;
see for example Chen, Hu, and Nualart \cite{LCHN17}.

\begin{lemma}[Spatial Stationarity]\label{lem:stat}
	Suppose $h\in\cup_{p>1}\mathcal{G}_p(\R^d)$, so that \eqref{SHE}
	has a unique random-field solution $u$ (Theorem \ref{th:exist}).
	Then, the random field $u\circ\theta_y$ has the same finite-dimensional
	distributions as $u$ for every $y\in\R^d$.
	In particular, for every $t\ge0$, the finite-dimensional
	distributions of $\{u(t\,,x+y)\}_{x\in\R^d}$ do not depend on $y\in\R^d$.
\end{lemma}

\begin{proof}
	The fact that \eqref{SHE} has a strong solution is another way to state
	that the transformation $\xi\mapsto u$ defines canonically a
	``solution map'' $\mathcal{S}$ via
	$u=\mathcal{S}(\xi)$, where we recall $\xi$ denotes space-time
	white noise. Recall also that the generalized Gaussian random field $\eta$ can be identified with a
	densely-defined isonormal Gaussian
	process $C_c(\R_+\times\R^d)\ni \varphi\mapsto\eta(\varphi)$
	via Wiener integrals as follows:
	\[
		\eta(\varphi) = \int_{\R_+\times\R^d}\varphi\,\d\eta
		\qquad\text{for all $\varphi\in L^2(\R_+\times\R^d)$}.
	\]
	Since $C_c(\R_+\times\R)\ni \varphi\mapsto\eta(\varphi)\in L^2(\P)$
	is a continuous linear mapping, the preceding identifies $\eta$ completely provided only
	that we prescribe $\eta(\varphi)$ for every $\varphi\in C_c(\R_+\times\R)$.
	In this way, we can  define a Gaussian noise $\eta_y$---one for every $y\in\R^d$---via
	\begin{equation}\label{xi_h}
		\eta_y(\varphi) = \int_{\R_+\times\R^d} \varphi(t\,,x-y)\,\eta(\d t\,\d x)
		\qquad\text{for all $\varphi\in C_c(\R_+\times\R^d)$.}
	\end{equation}
	It is easy to check covariances in order to see that $\eta_y(\varphi)$
	and $\eta(\varphi)$ have the same law; therefore, the noises
	$\eta$ and $\eta_y$ have the same law for every $y\in\R^d$. Also,
	it follows from the construction of the Walsh/It\^o stochastic integral that
	for all $t\ge0$, $x,y\in\R^d$, and Walsh-integrable random fields $\Psi$,
	\begin{equation}\label{WI}
		\int_{(0,t)\times\R^d}\Psi(s\,,z-y)\eta(\d s\,\d z)
		= \int_{(0,t)\times\R^d}\Psi(s\,,z)\,\eta_y(\d s\,\d z)
		\qquad\text{a.s.}
	\end{equation}
	This can be proved by standard approximation arguments, using only the fact that
	\eqref{WI} holds by \eqref{xi_h} when $\Psi$ is a simple random field;
	see Walsh \cite[Chapter 2]{Walsh}.
	
	Finally, we may combine \eqref{SHE} and \eqref{WI} in order to see that
	for all $t\ge0$ and $x,y\in\R^d$,
	\begin{align*}
		u(t\,,x+y) &= 1 + \int_{(0,t)\times\R^d} p_{t-s}(x+y-z)
			\sigma(u(s\,,z-y+y))\,\eta(\d s\,\d z)\\
		&=  1 + \int_{(0,t)\times\R^d} p_{t-s}(x-z)
			\sigma(u(s\,,z+y))\,\eta_y(\d s\,\d z)
			\qquad\text{a.s.}
	\end{align*}
	This proves that $u\circ\theta_y = \mathcal{S}(\eta_y)$ a.s.\ for every $y\in\R^d$, where
	we recall $\mathcal{S}$ denotes the solution map in \eqref{SHE}.
	Because $u$ is continuous,
	the preceding is another way to state the
	first assertion of the result. The second assertion follows from the first for elementary
	reasons.
\end{proof}

Let us mention also the following simple fact.

\begin{lemma}\label{lem:Var:erg}
	A stationary process $Y:=\{Y(x)\}_{x\in\R^d}$ is ergodic provided that
	\begin{equation}\label{cond:var}
		\lim_{N\to\infty}
		\Var\left( \fint_{[0,N]^d} \prod_{j=1}^k g_j(Y(x+\zeta^j))\,\d x\right) =0,
	\end{equation}
	for all integers $k\ge1$, every $\zeta^1,\ldots,\zeta^k\in\R^d$,
	and all Lipschitz-continuous functions $g_1,\ldots,g_k:\R\to\R$ that satisfy
	\begin{equation}\label{eq:WLOG}
		g_j(0)=0\quad\text{and}\quad
		\text{\rm Lip}(g_j)=1,
	\end{equation}
	for every $j=1,\ldots,k$.
\end{lemma}

\begin{proof}
	Suppose $g_1,\ldots,g_k:\R\to\R$ are non-constant,
	Lipschitz-continuous functions, but do not necessarily satisfy
	\eqref{eq:WLOG}. We first verify that \eqref{cond:var} holds
	for these $g_i$'s as well. Indeed, define
	\[
		\widetilde{g}_j(w) := \frac{g_j(w)-g_j(0)}{\text{\rm Lip}(g_j)}
		\qquad\text{for all $j=1,\ldots,k$ and $w\in\R$},
	\]
	and observe that $\widetilde{g}_1,\ldots,\widetilde{g}_k:\R\to\R$
	satisfy \eqref{eq:WLOG}, and hence \eqref{cond:var} holds
	when we replace every $g_i$ with $\tilde{g}_i$. It is easy to see that
	\begin{equation}\label{VarVar}
		\fint_{[0,N]^d} \prod_{j=1}^k g_j(Y(x+\zeta^j))\,\d x
		= \sum_{E\subseteq\{1,\ldots,k\}}
		\prod_{l\in E} g_l(0) \fint_{[0,N]^d} \prod_{j\in\{1,\ldots,k\}\setminus E}
		\text{\rm Lip}(g_j)\,\widetilde{g}_j(Y(x+\zeta^j))\,\d x,
	\end{equation}
	where a product over the empty set is identically defined as $1$.
	For example, when $k=2$, we have
	\begin{align*}
		&\fint_{[0,N]^d} g_1(Y(x+\zeta^1))g_2(Y(x+\zeta^2))\,\d x\\
		&\hskip1in= \fint_{[0,N]^d} \left[\text{\rm Lip}(g_1)\,\widetilde{g}_1(Y(x+\zeta^1))
			+ g_1(0)\right]\left[\text{\rm Lip}(g_2)\,\widetilde{g}_2(Y(x+\zeta^2))
			+ g_2(0)\right]\d x,
	\end{align*}
	which yields \eqref{VarVar} upon expanding the product inside the integral.
	
	Minkowski's inequality ensures that, for all random variables $X_1,\ldots,X_M\in L^2(\P)$,
	\[
		\Var(X_1+\cdots+X_M) \le \left(\sum_{i=1}^M \sqrt{\Var(X_i)}\right)^2
		\le M^2 \max_{1\le i\le M}\Var(X_i).
	\]
	Thus, we see from \eqref{VarVar} that
	\begin{align*}
		&\Var\left(\fint_{[0,N]^d} \prod_{j=1}^k g_j(Y(x+\zeta^j))\,\d x\right)\\
		&\le4^k  \max_{E\subseteq\{1,\ldots,k\}}
			\prod_{l\in E} g_l^2(0)\cdot
			\Var\left(\fint_{[0,N]^d} \prod_{j\in\{1,\ldots,k\}\setminus E}
			\text{\rm Lip}(g_j)\,\widetilde{g}_j(Y(x+\zeta^j))\,\d x\right)\\
		&\to0\quad\text{as $N\to\infty$},
	\end{align*}
	thanks to \eqref{cond:var}. This proves the assertion that if \eqref{cond:var}
	holds when $g_i$'s are Lipschitz and satisfy \eqref{eq:WLOG}, then
	\eqref{cond:var} continues to hold for non-constant, Lipschitz-continuous $g_i$'s, even when
	they do not satisfy \eqref{eq:WLOG}. And it is easy to see that ``non-constant'' can
	be removed from the latter assertion without changing its truth: We merely factor out
	of the variance the constant $g_i$'s, and relabel the remaining $g_j$'s, thus reducing the problem
	to the non-constant case.
	
	We now apply the preceding with $g_i$'s replaced with sines and cosines, in order
	to deduce from Chebyshev's inequality and stationarity that
	\[
		\lim_{N\to\infty}
		\fint_{[0,N]^d}\exp\left\{i\sum_{j=1}^k z_jY(x+\zeta^j)\right\}\d x
		=\E\left[\exp\left\{i\sum_{j=1}^kz_jY(\zeta^j)\right\}
		\right]\qquad\text{in $L^2(\P)$},
	\]
	for all $z_1,\ldots,z_k\in\R$ and $\zeta^1,\ldots,\zeta^k\in\R^d$.
	On the other hand, von-Neumann's (simpler) form of the 
	ergodic theorem \cite{Peterson} tells us that
	\[
		\lim_{N\to\infty}
		\fint_{[0,N]^d}\exp\left\{i\sum_{j=1}^k z_jY(x+\zeta^j)\right\}\d x
		= \E\left[ \left. \exp \left\{i\sum_{j=1}^k z_jY(\zeta^j)\right\}
		\ \right|\, \mathscr{I}\right]
		\qquad\text{in $L^2(\P)$,}
	\]
	where $\mathscr{I}$ denotes the invariant $\sigma$-algebra
	of $Y$. Equate the preceding two displays, and apply the inversion theorem of
	Fourier transforms, in order to see that
	every random vector of the form
	$(Y(\zeta^1)\,,\ldots,Y(\zeta^k))$ is independent of $\mathscr{I}$.
	This implies that $\mathscr{I}$ is independent of the $\sigma$-algebra generated
	by $Y$, and in particular $\mathscr{I}$ is independent of itself. This in turn
	proves the result.
\end{proof}

\section{Strong localization}\label{sec:localization}

In this section we refine a localization construction of Conus et al
\cite{CJKS2013} that works for a large
class of spatial correlation functions $f$ of the form \eqref{f=h*h}.

\begin{lemma}\label{lem:u^m,h}
	Choose and fix a real number $m>0$ and a function
	$h\in\mathcal{G}_p(\R^d)$ for some $p>1$.
	Then, the following stochastic integral equation has a predictable random-field
	solution $u^{(m,h)}$:
	\begin{equation}\label{eq:u^m,h}
		u^{(m,h)}(t\,,x) = 1 + \int_{(0,t)\times\mathbb{B}_{m\sqrt t}(x)}
		\bm{p}_{t-s}(x-y)\sigma\left( u^{(m,h)}(s\,,y)\right)\,\eta^{( h )}(\d s\,\d y),
	\end{equation}
	where $\eta^{(h)}$ was defined in \eqref{eta^h}.
	Moreover, $u^{(m,h)}$  is the only such solution that satisfies
	\begin{equation}\label{N:beta(u^m,h)}
		\mathcal{N}_{\beta,k}\left( u^{(m,h)}\right)<\infty\quad
		\text{whenever}\quad\text{$\beta \ge \Lambda_h\left(\frac{2(1-\varepsilon)^2}{%
		[z_k\text{\rm Lip}(\sigma)]^2}\right)$\quad for some $\varepsilon\in(0\,,1)$,}
	\end{equation}
	valid for every $k\ge2$.
\end{lemma}

\begin{proof}
	For every $n\in\mathbb{N}$, $t>0$, and $x\in\R^d$ define
	\[
		u_{n+1}^{(m, h )}(t\,,x) = 1 + \int_{(0,t)\times\mathbb{B}_{m\sqrt t}(x)}
		\bm{p}_{t-s}(x-y)\sigma\left( u_n^{(m, h )}(s\,,y)\right)\,\eta^{( h )}(\d s\,\d y),
	\]
	where  $u^{(m,h)}_0 \equiv1.$
	We can now repeat the last portions
	of the proof of Theorem \ref{th:exist} (see, in particular, \eqref{N(u)})
	in order to see that for every $k\ge2$,
	\begin{equation}\label{N(u^m,h)}\begin{split}
		\mathcal{N}_{\beta,k}\left(u_{n+1}^{(m, h )}\right) &\le 1 + z_k
			\mathcal{N}_{\beta,k}\left( \sigma(u_n^{(m,h)})\right)\sqrt{\frac12\int_{\R^d}
			\bm{v}_\beta(x) |f(x)|\,\d x}\\
		&\le 1+ z_k\left\{ |\sigma(0)| + \text{\rm Lip}(\sigma)\mathcal{N}_{\beta,k}
			\left( u_n^{(m,h)}\right)\right\}\sqrt{\frac12\int_{\R^d}
			\bm{v}_\beta(x) |f(x)|\,\d x},
	\end{split}\end{equation}
	and
	\begin{equation}\label{u^m,h:Picard}
		\mathcal{N}_{\beta,k}\left( u_{n+1}^{(m,h)} - u_n^{(m,h)}\right)
		\le z_k\text{\rm Lip}(\sigma)\mathcal{N}_{\beta,k}\left(
		u_n^{(m,h)}-u_{n-1}^{(m,h)}\right)\sqrt{\frac12\int_{\R^d}
		\bm{v}_\beta(x) \left(|h|*|\tilde{h}|\right)(x)\,\d x};
	\end{equation}
	see \eqref{u_n-u}. Since the final integral converges
	(Lemma \ref{lem:Dalang:Lp}), the rest of the proof
	follows by adapting the reasoning behind Theorem \ref{th:exist}
	to the present setting as well.
\end{proof}

\begin{lemma}\label{lem:u^m,h-u^m,h_n}
	Choose and fix a real number $m>0$ and a function
	$h\in\mathcal{G}_p(\R^d)$ for some $p>1$. Then,
	\[
		\mathcal{N}_{\beta,k}\left( u^{(m,h)} - u_n^{(m,h)}\right)\le
		\left[2 +  \frac{|\sigma(0)| + 1}{\text{\rm Lip}(\sigma)]}\right]\cdot
		\frac{(1-\varepsilon)^n}{\varepsilon}
		\qquad\text{for every $n\in\mathbb{Z}_+$},
	\]
	as long as $\beta\ge\Lambda_h(2(1-\varepsilon)^2/
	[z_k\text{\rm Lip}(\sigma)]^2)$
	for some $\varepsilon\in(0\,,1)$.
\end{lemma}

\begin{proof}
	In accord with \eqref{u^m,h:Picard}, the following holds for all $n\ge0$:
	\[
		\mathcal{N}_{\beta,k}\left( u_{n+1}^{(m,h)} - u_n^{(m,h)}\right)
		\le (1-\varepsilon)\mathcal{N}_{\beta,k}
		\left( u_n^{(m,h)} - u_{n-1}^{(m,h)}\right),
	\]
	whenever $\beta\ge\Lambda_h(2(1-\varepsilon)^2/
	[z_k\text{\rm Lip}(\sigma)]^2)$.
	Iterate to find that
	\begin{equation}\label{NNnn}
		\mathcal{N}_{\beta,k}\left( u_{n+1}^{(m,h)} - u_n^{(m,h)}\right)
		\le (1-\varepsilon)^n\mathcal{N}_{\beta,k}
		\left( u_1^{(m,h)} - u_0^{(m,h)}\right).
	\end{equation}
	Since $u_0^{(m,h)}\equiv1$, it follows readily from \eqref{N(u^m,h)} that
	\begin{align*}
		\mathcal{N}_{\beta,k}\left( u_1^{(m,h)} - u_0^{(m,h)}\right) &\le
			\mathcal{N}_{\beta,k}\left(u_1^{(m,h)}\right)+
			\mathcal{N}_{\beta,k}\left(u_0^{(m,h)}\right)\\
		&\le \mathcal{N}_{\beta,k}\left(u_1^{(m,h)}\right)+1\\
		&\le 2+ z_k\{ |\sigma(0)| + 1\}\sqrt{\frac12\int_{\R^d}
			\bm{v}_\beta(x) |f(x)|\,\d x}\\
		&\le 2 +  \frac{|\sigma(0)| + 1}{\text{\rm Lip}(\sigma)]}(1-\varepsilon)\\
		&\le 2 +  \frac{|\sigma(0)| + 1}{\text{\rm Lip}(\sigma)]},
	\end{align*}
	provided that $\beta\ge\Lambda_h(2(1-\varepsilon)^2/[z_k\text{\rm Lip}(\sigma)]^2)$;
	see also \eqref{eq:N(u_n+1)}.
	This, \eqref{NNnn}, and the defining property of $u^{(m,h)}$ together yield
	\[
		\mathcal{N}_{\beta,k}\left( u^{(m,h)} - u_n^{(m,h)}\right)
		\le \sum_{k=n}^\infty
		\mathcal{N}_{\beta,k}\left(
		u_{k+1}^{(m,h)}-u_k^{(m,h)}\right)
		\le \left[2 +  \frac{|\sigma(0)| + 1}{\text{\rm Lip}(\sigma)]}\right]
		\sum_{k=n}^\infty(1-\varepsilon)^k,
	\]
	and hence the lemma.
\end{proof}

\begin{lemma}\label{int_0^t}
	Choose and fix a a function
	$h\in\mathcal{G}_p(\R^d)$ for some $p>1$.
	Then, there exists a real number $c = c(d)>0$, independent of $f$, such that
	\[
		\int_0^{2t}\d s
		\int_{\R^d}\d w\ \bm{p}_s(w)|f(w)|
		\le c\int_0^{\sqrt t}\bar{f}(r)\omega_d(r)\,\d r,
	\]
	simultaneously for all $t\in(0\,,1)$.
\end{lemma}

\begin{proof}
	First of all, let us observe that
	\begin{align*}	
		\int_0^{2t}\d s
			\int_{\R^d}\d w\ \bm{p}_s(w)|f(w)|
			&\lesssim\int_0^{2t}s^{-d/2}\,\d s\int_0^\infty r^{d-1}\,\d r\
			\exp\left(-\frac{r^2}{2s}\right)\bar{f}(r)\\
		&=\int_0^\infty r^{d-1}\bar{f}(r)\,\d r\int_0^{2t}\d s\
			s^{-d/2}\exp\left(-\frac{r^2}{2s}\right),
	\end{align*}
	where the implied constant depends only on $d$.
	Therefore, a change of variables yields
	\begin{align*}	
		\int_0^{2t}\d s
			\int_{\R^d}\d w\ \bm{p}_s(w)|f(w)|
			&\lesssim\int_0^\infty r\bar{f}(r)\,\d r\int_0^{2t/r^2}\d v\
			v^{-d/2}\exp\left(-\frac{1}{2v}\right)\\
		&= Q_1 + Q_2,
	\end{align*}
	where
	\begin{align*}
		Q_1 &:= \int_{\sqrt t}^\infty r\bar{f}(r)\,\d r\int_0^{2t/r^2}\d v\
			v^{-d/2}\exp\left(-\frac{1}{2v}\right),\\
		Q_2 &:= \int_0^{\sqrt t} r\bar{f}(r)\,\d r\int_0^{2t/r^2}\d v\
			v^{-d/2}\exp\left(-\frac{1}{2v}\right),
	\end{align*}
	and the implied constant depends only on $d$.
	Because $\bar{f}$ is monotonically non-increasing, we apply l'H\^opitals rule
	to $\int_0^\delta v^{-d/2}\exp(-1/(2v))\,\d v$ (as $\delta\downarrow0$) in order
	to see that
	\[
		Q_1\lesssim
		\bar{f}\left(\sqrt t\right)\int_{\sqrt t}^\infty r\left(\frac{t}{r^2}\right)^{-(d-4)/2}
		\exp\left( -\frac{r^2}{4t}\right)\d r
		\lesssim \frac{\bar{f}\left(\sqrt t\right)}{t^{(d-4)/2}}\int_{\sqrt t}^\infty r^{d-3}
		\exp\left( -\frac{r^2}{4t}\right)\d r,
	\]
	with no parameter dependencies other than dependency on $d$.
	Therefore, a change of variables yields
	$Q_1\lesssim t\bar{f}(\sqrt t)\propto\bar{f}(\sqrt t)\int_0^{\sqrt t}r\,\d r$ uniformly for all $t>0$.
	Since $\bar{f}$ is monotonically non-increasing, this
	in turn implies that
	\[
		Q_1\lesssim \int_0^{\sqrt t}\bar{f}(r)\omega_d(r)\,\d r,
	\]
	simultaneously for every $t>0$. This is our final estimate for $Q_1$.

	Our estimate for $Q_2$ proceeds by studying the cases $d = 1$, $d=2$ and $d\ge3$ separately.
	First consider the case that $d\ge3$. In that case,
	\[
		Q_2 \le
		\int_0^{\sqrt t} r\bar{f}(r)\,\d r\int_0^\infty\d v\
		v^{-d/2}\exp\left(-\frac{1}{2v}\right)
		\propto\int_0^{\sqrt t}r\bar{f}(r)\,\d r
		=\int_0^{\sqrt t}\bar{f}(r)\omega_d(r)\,\d r,
	\]
	where the implied constant depends only on $d$.
	On the other hand, if $d=2$, then
	\begin{align*}
		Q_2&\le\int_0^{\sqrt t} r\bar{f}(r)\,\d r\int_0^{2t/r^2}\d v\
			v^{-1}\exp\left(-\frac{1}{2v}\right)
			\lesssim\int_0^{\sqrt t} r\bar{f}(r)
			\left( 1 + \int_1^{2t/r^2}\frac{\d v}{v}\right)\d r\\
		&\lesssim\int_0^{\sqrt t} \bar{f}(r)  r\log_+\left( \frac{\sqrt t}{r}\right)\d r
			\le\int_0^{\sqrt t}\bar{f}(r)  r\log_+(1/r)\,\d r=
			\int_0^{\sqrt t}\bar{f}(r)\omega_2(r)\,\d r,
	\end{align*}
	uniformly for all $0 < t < 1$.
	
    	If $d=1$, then
	\begin{align*}
		\int_0^{\sqrt t} r\bar{f}(r)\,\d r\int_0^{2t/r^2}\d v\
			v^{-d/2}\exp\left(-\frac{1}{2v}\right) &\le
			\int_0^{\sqrt t} r\bar{f}(r)\,\d r\int_0^{2t/r^2}\d v\
			v^{-1/2}\exp\left(-\frac{1}{2v}\right)\\
		&\lesssim\int_0^{\sqrt t}r\bar{f}(r)\,\d r\left( 1 +
			\int_1^{2t/r^2}\frac{\d v}{\sqrt v}\right)\\
		&\lesssim \int_0^{\sqrt t}\bar{f}(r)\,\d r
			= \int_0^{\sqrt t}\bar{f}(r)\omega_1(r)\,\d r,
	\end{align*}
   	 uniformly for all $0 < t < 1$.

    	Combine the bounds for $Q_1$ and $Q_2$ in order to deduce the lemma.
\end{proof}

Before we proceed further with our technical estimates,
let us define a dimension-dependent \emph{gauge function} $\gamma_d$ as follows:
\begin{equation}\label{gamma_d}
	\gamma_d(t) := \sup_{s\in(0,1)}\left( t\bar{\bm{v}}_1(s) \wedge
	\int_0^s\bar{f}(r)\omega_d(r)\,\d r\right)
	\qquad\text{for all $t>0$},
\end{equation}
where, for every $\lambda,t>0$,
\begin{equation}\label{barbar}
	\bar{\bm{v}}_\lambda(s) = \int_0^\infty\e^{-\lambda t}
	\bar{\bm{p}}_t(s)\,\d t \quad \text{and} \quad
	\bar{\bm{p}}_t(s):= (2\pi t)^{-d/2}\exp(-s^2/(2t))
	\qquad[s>0]
\end{equation}
are analogous to $\bm{v}_\lambda$ [see \eqref{HRD}] and $\bm{p}_t$
[see \eqref{p}], but are now functions on $(0\,,\infty)$.

The gauge function $\gamma_d$ will play an important role in the sequel. Therefore, let us
identify some of its first-order properties first.

\begin{lemma}\label{lem:gamma_d}
	$\gamma_1(t)\asymp t$ uniformly for all $t\in[0\,,1]$.
	Moreover,
	\begin{equation}\label{gamma_d:behavior}
		\lim_{t\to0^+}\gamma_d(t)=0\quad\text{and}\quad
		\liminf_{t\to0^+} \frac{\gamma_d(t)}{t}=\infty\quad
		\text{for every $d\ge 2$.}
	\end{equation}
\end{lemma}

\begin{proof}
	First of all, $\gamma_1(t)\lesssim t$ uniformly for all $t\ge0$
	because $\bar{\bm{v}}_1(r)\le\bar{\bm{v}}_1(0)=2^{-1/2}<\infty$.
	Next, let us choose and fix $\delta\in(0\,,1)$, and note that
	\begin{align*}
		\gamma_d(t) & =
			\sup_{s\in(0,\delta)}\left( t\bar{\bm{v}}_1(s) \wedge
			\int_0^s\bar{f}(r)\omega_d(r)\,\d r\right)
			\vee \sup_{s\in(\delta,1)}\left( t\bar{\bm{v}}_1(s) \wedge
			\int_0^s\bar{f}(r)\omega_d(r)\,\d r\right)\\
		&\le\int_0^{\delta}\bar{f}(r)\omega_d(r)\,\d r \vee t\bar{\bm{v}}_1(\delta).
	\end{align*}
	First let $t\downarrow 0$ and then let $\delta\downarrow 0$ in order to see that
	$\lim_{t\to0^+}\gamma_d(t)=0$.
	This proves half of the assertion \eqref{gamma_d:behavior}.
	Finally, we observe that for every $\delta\in(0\,,1)$,
	\[
		\gamma_d(t)\ge t\bar{\bm{v}}_1(\delta)\wedge
		\int_0^\delta\bar{f}(r)\omega_d(r)\,\d r\sim t\bar{\bm{v}}_1(\delta)
		\qquad\text{as $t\to0^+$.}
	\]
	Since $\bar{\bm{v}}_1(0+)<\infty$ iff $d=1$,
	this proves that $\gamma_1(t)\gtrsim t$ uniformly for
	all $t\in[0\,,1]$, and also
	completes the remaining half of \eqref{gamma_d:behavior}.
\end{proof}

Our next lemma is a strong localization result. It might help to
recall the definition of the gauge function $\gamma_d$ from \eqref{gamma_d}.

\begin{lemma}\label{lem:N(u-um)}
	Choose and fix a a function
	$h\in\mathcal{G}_p(\R^d)$ for some $p>1$.
	Then, for every $\varepsilon\in(0\,,1)$ and $k \ge 2$,
	there exists $c_h=c_h(\varepsilon\,,d\,,k)>0$
	such that
	\begin{equation}\label{eq:N(u-um)}
		\mathcal{N}_{\beta,k}\left( u - u^{(m,h)}\right) \le
		c_h\sqrt{\gamma_d\left(m^{d-2}\e^{-m^2/2}\right)},
	\end{equation}
	uniformly for all $\beta\ge1\vee\Lambda_h(2(1-\varepsilon)^2/[z_k\text{\rm Lip}(\sigma)]^2)$
	and $m\ge1$. Moreover, if $H\in\mathcal{G}_p(\R^d)$ satisfies $|h|\le H$,
	then $c_h\le c_H$.
\end{lemma}

We mentioned right before Lemma \ref{lem:gamma_d} that
the function $\gamma_d$ is a gauge function
that plays an important role in our analysis. Lemma
\ref{lem:N(u-um)} explains
the choice of the word ``gauge function'' by showing that
$\gamma_d$ indeed shows how well we may approximate
the solution to \eqref{SHE} by the strongly-localized process $u^{(m,h)}$.
In fact, we may unscramble the preceding to see that, in the notation of
Lemma \ref{lem:N(u-um)},
\[
	\sup_{x\in\R^d}\E\left( \left| u(t\,,x) - u^{(m,h)}(t\,,x)\right|^k\right)
	\le c^k\exp\left\{ \left[1\vee k\Lambda_h\left(
	\frac{2(1-\varepsilon)^2}{[z_k\text{\rm Lip}(\sigma)]}
	\right)\right]t\right\}
	\left[\gamma_d\left(m^{d-2}\e^{-m^2/2}\right)\right]^{k/2}.
\]

\begin{proof}[Proof of Lemma \ref{lem:N(u-um)}]
	Choose and fix an arbitrary $k\ge 2$.
	By monotonicity, it suffices to consider only the case that
	\begin{equation}\label{def:beta}
		\beta = 1\vee
		\Lambda_h\left( \frac{2(1-\varepsilon)^2}{[z_k\text{\rm Lip}(\sigma)]^2}\right),
	\end{equation}
	an identity which we assume holds throughout the proof.
	
	For every $t>0$ and $x\in\R^d$ we can write
	\[
		\mathcal{D}(t\,,x) := u(t\,,x) - u^{(m,h)}(t\,,x)= I_1 + I_2,
	\]
	where
	\begin{align*}
		I_1&:=\int_{(0,t)\times\mathbb{B}_{m\sqrt t}(x)}
			\bm{p}_{t-s}(x-y)\left[\sigma\left(u(s\,,y)\right)
			-\sigma\left( u^{(m,h)}(s\,,y)\right)\right]\eta^{(h)}(\d s\,\d y),\\
		I_2 &:=\int_{(0,t)\times[\mathbb{B}_{m\sqrt t}(x)]^c}
			\bm{p}_{t-s}(x-y)\sigma\left(u(s\,,y)\right)
			\eta^{(h)}(\d s\,\d y).
	\end{align*}
	We bound the $k$-th moments of $I_1$ and $I_2$ in this order.
	
	First, note that the method of proof of Lemma \ref{lem:Young} yields
	\[
		\|I_1\|_k^2\le z_k^2\int_0^t\d s\int_{\mathbb{B}_{m\sqrt t}(x)}\d y
		\int_{\mathbb{B}_{m\sqrt t}(x)} \d z\
		\bm{p}_{t-s}(x-y)\bm{p}_{t-s}(x-z)
		\|\mathcal{A}(s\,,y)\|_k\|\mathcal{A}(s\,,z)\|_k|f(y-z)|,
	\]
	where $\mathcal{A}(s\,,y) := \sigma(u(s\,,y))
	-\sigma( u^{(m,h)}(s\,,y)).$
	Because $\sigma$ is Lipschitz continuous, it follows that
	$\|\mathcal{A}(s\,,y)\|_k\le\text{\rm Lip}(\sigma)\|\mathcal{D}(s\,,y)\|_k$, and hence
	\begin{align*}
		\|I_1\|_k^2 &\le z_k^2[\text{\rm Lip}(\sigma)]^2
			\left[\sup_{\substack{s\le t\\y\in\R^d}}\e^{-\beta s}
			\left\|\mathcal{D}(s\,,y)\right\|_k\right]^2
			\int_0^t \e^{2\beta s}\,\d s\int_{\R^d}\d y
			\int_{\R^d} \d z\
			\bm{p}_{t-s}(y)\bm{p}_{t-s}(z)|f(y-z)|\\
		&\le z_k^2\e^{2\beta t}[\text{\rm Lip}(\sigma)]^2
			\left[\sup_{\substack{s\le t\\y\in\R^d}}\e^{-\beta s}\,
			\left\| \mathcal{D}(s\,,y)\right\|_k\right]^2
			\int_0^\infty \e^{-2\beta s}\,\d s \int_{\R^d} \d w\ \bm{p}_{2s}(w)|f(w)|\\
		&= \frac12 z_k^2\e^{2\beta t}[\text{\rm Lip}(\sigma)]^2
			\left[\sup_{\substack{s\le t\\y\in\R^d}}\e^{-\beta s}
			\left\| \mathcal{D}(s\,,y)\right\|_k\right]^2
			\int_{\R^d}\bm{v}_\beta(w)\left( |h|*|\tilde{h}|\right)(w)\,\d w.
	\end{align*}
	Apply the definition of $\beta$---see \eqref{def:beta}---in order to see that
	\begin{equation}\label{EI_12}
		\e^{-\beta t}\|I_1\|_k\le(1-\varepsilon)
		\sup_{\substack{s\le t\\y\in\R^d}}\e^{-\beta s}
		\left\| \mathcal{D}(s\,,y)\right\|_k.
	\end{equation}
	This yields the desired bound for $\E(|I_1|^k)$.
	
	Next we estimate $\E(|I_2|^k)$ as follows:
	\[
		\|I_2\|_k^2 \le z_k^2\int_0^t\d s\int_{\|y\|>m\sqrt t}\d y\int_{\|z\|>m\sqrt t}\d z\
		\bm{p}_{t-s}(y)\bm{p}_{t-s}(z)\|\mathcal{B}(s\,,y)\|_k
		\|\mathcal{B}(s\,,z)\|_k |f(y-z)|,
	\]
	where $\mathcal{B}(s\,,y) := \sigma( u(s\,,y)).$ Apply the
	Lipschitz continuity of $\sigma$ to see that
	\[
		\|\mathcal{B}(s\,,y)\|_k \le |\sigma(0)| + \text{\rm Lip}(\sigma)
		\left\| u(s\,,y)\right\|_k\\
		\le |\sigma(0)| + \frac{\text{\rm Lip}(\sigma)}{\varepsilon}
		\left[ 1 +
		\frac{|\sigma(0)|}{\text{\rm Lip}(\sigma)}\right]
		\e^{\beta s},
	\]
	thanks to the moment bound of $u$ in Theorem \ref{th:exist},
	and owing to the definition of $\beta$; see \eqref{def:beta}. This yields
	$\|\mathcal{B}(s\,,y)\|_k \lesssim\varepsilon^{-1}\exp(\beta s)$,
	where the implied constant depends only on $\sigma$. Thus, we see that
	\begin{align*}
		\|I_2\|_k^2 &\lesssim\frac{z_k^2}{\varepsilon^2}
			\int_0^t\e^{2\beta s}\,\d s\int_{\|y\|>m\sqrt t}\d y\int_{\|z\|>m\sqrt t}\d z\
			\bm{p}_{t-s}(y)\bm{p}_{t-s}(z)|f(y-z)|\\
		&\le\frac{z_k^2\e^{2\beta t}}{\varepsilon^2}
			\int_0^t\e^{-2\beta s}\,\d s\int_{\|y\|>m\sqrt t}\d y\int_{\|z\|>m\sqrt t}\d z\
			\bm{p}_s(y)\bm{p}_s(z)|f(y-z)|\\
		&= \frac{z_k^2\e^{2\beta t}}{\varepsilon^2}(W_1+W_2),
	\end{align*}
	where the implied constant depends only on $\sigma$, and
	\begin{align*}
		W_1 &:= \int_0^t\e^{-2\beta s}\,\d s\int_{\|y\|>m\sqrt t}\d y\int_{\substack{\|z\|>m\sqrt t\\
			\|y-z\|<1}}\d z\
			\bm{p}_s(y)\bm{p}_s(z)|f(y-z)| ,\\
		W_2 &:= \int_0^t\e^{-2\beta s}\,\d s\int_{\|y\|>m\sqrt t}\d y\int_{\substack{\|z\|>m\sqrt t\\
			\|y-z\|>1}}\d z\
			\bm{p}_s(y)\bm{p}_s(z)|f(y-z)| .
	\end{align*}

	In order to evaluate/estimate $W_1$, let us first choose and fix some $s\in(0\,,t)$ and observe that
	\begin{align*}
		\int_{\|y\|>m\sqrt t}\d y\int_{\substack{\|z\|>m\sqrt t\\
			\|y-z\|<1}}\d z\
			\bm{p}_s(y)\bm{p}_s(z)|f(y-z)| &\le
			\bar{\bm{p}}_s\left(m\sqrt t\right)\int_{\mathbb{B}_1}|f(w)|\,\d w
			\int_{\|y\|>m\sqrt t}\bm{p}_s(y)\,\d y\\
		&\lesssim m^{d-2}\e^{-m^2/2}\bar{\bm{p}}_s\left(m\sqrt t\right)\int_{\mathbb{B}_1}|f(w)|\,\d w,
	\end{align*}
	where the implied constant depends only on $d$. In the last line we have used
	scaling, and the well-known fact that
	\begin{equation}\label{Mills}
		\int_{\|y\|>\nu}\bm{p}_1(y)\,\d y =
		\frac{(2+o(1))}{2^{d/2}\Gamma(d/2)}\nu^{d-2}\e^{-\nu^2/2}
		\qquad\text{as $\nu\to\infty$}.
	\end{equation}
	According to Lemma \ref{lem:PD} (see, in particular, \eqref{eq:PD:int}), the preceding integral is finite. Therefore,
	\[
		W_1 \lesssim m^{d-2}\e^{-m^2/2}\bar{\bm{v}}_{2\beta}\left(m\sqrt t\right)
		\int_{\mathbb{B}_1}|f(w)|\,\d w\le
		m^{d-2}\e^{-m^2/2}\bar{\bm{v}}_{2\beta}\left(m\sqrt t\right)
		\int_{\mathbb{B}_1}\left( |h|*|\tilde{h}|\right)(x)\,\d w,
	\]
	where the implied constant depends only on $d$.
	Next we estimate $W_2$.
	
	Choose and fix $s\in(0\,,t)$ and note that, owing to \eqref{Mills},
	\begin{align*}
		\int_{\|y\|>m\sqrt t}\d y\int_{\substack{\|z\|>m\sqrt t\\
			\|y-z\|>1}}\d z\
			\bm{p}_s(y)\bm{p}_s(z)|f(y-z)| &\le
			\sup_{\|w\|>1}|f(w)|\left(\int_{\|y\|>m\sqrt t}\bm{p}_s(y)\,\d y\right)^2\\
		&\lesssim m^{2(d-2)}\e^{-m^2}\sup_{\|w\|>1}|f(w)|\\
		&\le \bar{f}(1)m^{2(d-2)}\e^{-m^2};
	\end{align*}
	see \eqref{bar(f)} for the latter notation.
	Thus, $W_2 \lesssim m^{2(d-2)}\e^{-m^2} \beta^{-1}\bar{f}(1),$
	where the implied constant depends only on $d$. This and the above inequality for
	$W_1$ together yield the following estimate of $\E(|I_2|^k)$:
	\begin{equation}\label{EI_22}
		\e^{-2\beta t}\|I_2\|_k^2 \lesssim\left(\frac{K(h)}{\varepsilon}\right)^2\left[
		m^{d-2}\e^{-m^2/2}\bar{\bm{v}}_{2\beta}\left(\sqrt t\right) +
		\frac{1}{\beta}m^{2(d-2)}\e^{-m^2}
		\right],
	\end{equation}
	where the implied constant depends only on $(d\,,k)$
	and
	\begin{equation}\label{K(h)}
		1+\int_{\mathbb{B}_1}|f(w)|\,\d w +  \bar{f}(1)
		\le 1+\int_{\mathbb{B}_1}\left( |h|*|\tilde{h}|\right)(x)\,
		\d w +  \bar{f}(1) =: K(h)<\infty;
	\end{equation}
	see \eqref{bar(f)}
	and Lemma \ref{lem:PD}.
	This and \eqref{EI_12} together yield a real number $C=C(d\,,k)>0$ such that
	\begin{equation}\label{prelim:ED}\begin{split}
		&\e^{-\beta t}\left\| \mathcal{D}(t\,,x) \right\|_k\le
			(1-\varepsilon)
			\sup_{\substack{s\le t\\y\in\R^d}}\left[\e^{-\beta s}
			\left\| \mathcal{D}(s\,,y)\right\|_k\right]
			+\e^{-\beta t}\|I_2\|_k\\
		&\le (1-\varepsilon)
			\sup_{\substack{s\le t\\y\in\R^d}}\left[\e^{-\beta s}
			\left\| \mathcal{D}(s\,,y)\right\|_k\right]
			+\frac{CK(h)}{\varepsilon}\left[
			m^{d-2}\e^{-m^2/2}\bar{\bm{v}}_{2\beta}\left(\sqrt t\right) +
			\frac{1}{\beta}m^{2(d-2)}\e^{-m^2}
			\right]^{1/2}\\
		&\le (1-\varepsilon)
			\sup_{\substack{s\le t\\y\in\R^d}}\left[\e^{-\beta s}
			\left\| \mathcal{D}(s\,,y)\right\|_k \right]
			+\frac{CK(h)}{\varepsilon}\left[
			m^{d-2}\e^{-m^2/2}\bar{\bm{v}}_1\left(\sqrt t\right) + m^{2(d-2)}\e^{-m^2}
			\right]^{1/2},
	\end{split}\end{equation}
	simultaneously for all $t>0$, $m\ge1$, and $x\in\R^d$. [We have also used the fact that
	$\beta\ge1$ in the last line; see \eqref{def:beta}.]
	In particular, uniformly for every $t\ge1$,
	\begin{align}\notag
		\e^{-\beta t} \left\| \mathcal{D}(t\,,x) \right\|_k \le&
			(1-\varepsilon)
			\sup_{\substack{s\le t\\y\in\R^d}}\left[\e^{-\beta s}
			\left\| \mathcal{D}(s\,,y) \right\|_k \right]
			+ \frac{CK(h)}{\varepsilon}\left[
			m^{d-2}\e^{-m^2/2}\bar{\bm{v}}_1(1) + m^{2(d-2)}\e^{-m^2}
			\right]^{1/2}\\\label{ED0}
		\le & (1-\varepsilon)
			\sup_{\substack{s\le t\\y\in\R^d}}\left[\e^{-\beta s}
			\left\| \mathcal{D}(s\,,y) \right\|_k \right]
			+ \frac{C'K(h)}{\varepsilon}\, m^{(d-2)/2}\e^{-m^2/4},
	\end{align}
	where $C$ and $C'$ depend only on $(d\,, k)$.
	
	Next, we proceed by estimating $\E(|I_2|^k)$ using a different idea than the one that
	led to \eqref{EI_22}, and using that different idea in \eqref{prelim:ED} instead. This idea works
	well when $t$ is small.
		
	Just as before,
	\begin{align*}
		\|I_2\|_k^2 &\le\int_0^t\d s\int_{\R^d}\d y\int_{\R^d}\d z\
			\bm{p}_{t-s}(y)\bm{p}_{t-s}(z)\|\mathcal{B}(s\,,y)\|_k
			\|\mathcal{B}(s\,,z)\|_k|f(y-z)|\\
		&\lesssim\frac{\e^{2\beta t}}{\varepsilon^2}\int_0^t\e^{-2\beta s}\,\d s
			\int_{\R^d}\d y\int_{\R^d}\d z\
			\bm{p}_s(y)\bm{p}_s(z)|f(y-z)|,
	\end{align*}
	where the implied constant is universal. Reorganize the integral and use the semigroup
	property of $\bm{p}$ (as we have done in the preceding lemmas several times)
	in order to see that
	\[
		\e^{-2\beta t}\|I_2\|_k^2 \lesssim\frac{1}{%
		\varepsilon^2}\int_0^{2t}\d s
		\int_{\R^d}\d w\ \bm{p}_s(w)|f(w)|,
	\]
	where the implied constant is still universal. Now apply Lemma \ref{int_0^t} in order to see that
	\[
		\e^{-2\beta t}\| I_2\|_k^2\lesssim\frac{1}{%
		\varepsilon^2}\int_0^{\sqrt t}\bar{f}(r)\omega_d(r)\,\d r
		\qquad\text{uniformly for all $t\in(0\,,1)$},
	\]
	where the implied constant depends only on $d$ and $k$, and $\bar{f}$
	was defined in \eqref{bar(f)}. Use this in \eqref{prelim:ED}
	instead of \eqref{EI_22} in order to see that there exists a real number
	$C=C(d\,,h)>0$ such that
	\begin{align*}
		\e^{-\beta t}\left\| \mathcal{D}(t\,,x)\right\|_k&\le
			(1-\varepsilon)
			\sup_{\substack{s\le t\\y\in\R^d}}\left[\e^{-\beta s}
			\left\| \mathcal{D}(s\,,y)\right\|_k\right]
			+\e^{-\beta t}\| I_2\|_k\\
		&\le (1-\varepsilon)
			\sup_{\substack{s\le t\\y\in\R^d}}\left[\e^{-\beta s}
			\left\| \mathcal{D}(s\,,y) \right\|_k \right]
			+\frac{C}{\varepsilon}\left[\int_0^{\sqrt t}\bar{f}(r)\omega_d(r)\,\d r
			\right]^{1/2},
	\end{align*}
	simultaneously for all $t\in(0\,,1)$. Combine this with \eqref{prelim:ED} to see that
	\begin{align*}
		\e^{-\beta t}\left\| \mathcal{D}(t\,,x)\right\|_k
		\le& (1-\varepsilon)
			\sup_{\substack{s\le t\\y\in\R^d}}\left[\e^{-\beta s}
			\left\| \mathcal{D}(s\,,y) \right\|_k \right]\\
			&+\frac{CK(h)}{\varepsilon}
			\min\left\{ m^{d-2}
			\e^{-m^2/2}\bar{\bm{v}}_1\left(\sqrt t\right)
			~,~ \int_0^{\sqrt t}\bar{f}(r)\omega_d(r)\,\d r\right\}^{1/2},
	\end{align*}
	uniformly for all $t\in(0\,,1)$. This and the definition \eqref{gamma_d} of
	the function $\gamma_d$ together yield
	\begin{equation}\label{comb1}
		\sup_{\substack{s\in(0,1)\\y\in\R^d}}\left[\e^{-\beta s}
		\left\| \mathcal{D}(s\,,y) \right\|_k \right] \le
		\frac{CK(h)}{\varepsilon}\sqrt{\gamma_d\left( m^{d-2}\e^{-m^2/2}\right)}.
	\end{equation}
	Also, we can read off from the above and \eqref{ED0} that
	for all $t\ge1$ and $x\in\R^d$,
	\begin{align*}
		\e^{-\beta t}\left\| \mathcal{D}(t\,,x) \right\|_k \le&
			\frac{CK(h)}{\varepsilon}\sqrt{\gamma_d\left( m^{d-2}\e^{-m^2/2}\right)}
			+ (1-\varepsilon)
			\sup_{\substack{s\in[1,t]\\y\in\R^d}}\left[\e^{-\beta s}
			\left\| \mathcal{D}(s\,,y) \right\|_k \right]\\
		&+ \frac{C'K(h)}{\varepsilon}\,m^{(d-2)/2}\e^{-m^2/4},
	\end{align*}
	where $C'=C'(d\,,k).$
	According to Lemma \ref{lem:gamma_d},
	$\gamma_d(a)\gtrsim a$ for all $a\in(0\,,1)$. Therefore,
	the third term on the right-hand side is dominated by a constant multiple
	of the first term, where the constant depends only on $(d\,,k)$. Thus,
	we find that
	\begin{equation}\label{comb2}
		\sup_{\substack{t\ge1\\
		x\in\R^d}}\left[\e^{-\beta t}
		\left\| \mathcal{D}(t\,,x) \right\|_k \right] \le
		\frac{C''K(h)}{\varepsilon}
		\sqrt{\gamma_d\left( m^{d-2}\e^{-m^2/2}\right)},
	\end{equation}
	where $C''=C''(d\,,k)$.
	We deduce \eqref{eq:N(u-um)} by combining \eqref{comb1} and \eqref{comb2}.
	Finally, suppose $H\in\mathcal{G}_p(\R^d)$ satisfies $|h|\le H$. Then,
	\eqref{K(h)} implies readily that $K(h)\le K(H)$, which implies the remaining assertion
	that $c_h\le c_H$.
\end{proof}

\section{A Poincar\'e-type inequality}\label{sec:Poincare}
The main result of this chapter is a type of
Poincar\'e-like inequality for the occupation measure of
$u(t)$, where $u$ solves \eqref{SHE}, and $t>0$ is fixed but otherwise arbitrary.
This Poincar\'e-type inequality
is the main technical innovation of the paper. We will see that,
among other things,
our Poincar\'e-type inequality implies the desired
spatial ergodicity of $u$.

\[
	\text{%
	Henceforth, we assume that the conditions of Theorem \ref{th:exist} are met.
	}
\]
In particular, we have also chosen and fixed a function $h\in\mathcal{G}_p(\R^d)$
for a fixed $p>1$.

Finally, in order to state our Poincar\'e-type inequality, let us recall the gauge function
$\gamma_d$ from \eqref{gamma_d}, and introduce a second gauge function,
\begin{equation}\label{tau_d}
	\tau_d(a) := \inf_{r>1}\left[ ar^d + \int_{\|w\|>r}|h(w)|^2\,\d w\right]
	\qquad\text{for all $a\ge0$}.
\end{equation}

\begin{theorem}[A Poincar\'e inequality]\label{th:Poincare}
	Choose and fix a real  $T>0$ and an integer $k\ge1$. Then,
	\begin{equation}\label{eq:Poincare}\begin{split}
		&\sup_{t\in[0,T]}\Var\left(\fint_{[0,N]^d} \prod_{l=1}^k g_l
			(u(t\,,x + \zeta^l))\,\d x\right) \\
		&\hskip1.5in\lesssim\inf_{a\in(0, c_d)}\left[
			\tau_d\left( \frac{|\log(1/a)|^{3d/2}+
			\max_{1\le j\le k}\|\zeta^j\|^d}{N^d}\right)
			+ \sqrt{\gamma_d(a)}\right],
	\end{split}\end{equation}
	where
	\begin{align}\label{c_d}
		c_d &=\begin{cases}
			\e^{-1/2} & \text{if $d \leq 3$}, \\
			(d - 2)^{-1+(d/2)}\e^{1-(d/2)} & \text{if $d \ge 4$},
		\end{cases}
	\end{align}
	uniformly for every real number $N>1$, all Lipschitz-continuous functions
	$g_1,\ldots,g_k:\R\to\R$ that satisfy \eqref{eq:WLOG}, and
	every $\zeta^1,\ldots,\zeta^k\in\R^d$.
\end{theorem}

Before we prove Theorem \ref{th:Poincare}, we would like
to explain why this is a Poincar\'e-type inequality. In order to see that
consider the special case that $k=1$, and define for all reals $N\ge1$ and $t>0$,
the spatial occupation measure $\mu_{t,N}$ of the restriction of
$x\mapsto u(t\,,x)$ to $[0\,,N]^d$ via
\[
	\langle g\,,\mu_{t,N}\rangle := \int_{\R^d}g\,\d\mu_{t,N} := \fint_{[0,N]^d}g(u(t\,,x))\,\d x.
\]
Theorem \ref{th:Poincare} then says that for every $T>0$,
\begin{equation}\label{eq:P}
	\Var\langle g\,,\mu_{t,N}\rangle \lesssim
	\inf_{a\in(0,c_d)}\left[
	\tau_d\left( \frac{|\log(1/a)|^{3d/2}}{N^d}\right)
	+ \sqrt{\gamma_d(a)}\right],
\end{equation}
uniformly for all $t\in[0\,,T]$, $N\ge1$, and Lipschitz-continuous
functions $g:\R\to\R$ that satisfy $g(0)=0$ and $\text{\rm Lip}(g)=1$.
If $g$ is constant then the preceding is a trivial bound. Else,
we may replace $g$ by $(g-g(0))/\text{\rm Lip}(g)$ in \eqref{eq:P}
in order to see that for every $T>0$ there exists a constant $C_T
= C_T(\sigma\,,h\,,d)>0$ such that
\begin{equation}\label{eq:P:1}
	\Var\langle g\,,\mu_{t,N}\rangle \le C_T\inf_{a\in(0, c_d )}\left[
	\tau_d\left( \frac{|\log(1/a)|^{3d/2}}{N^d}\right)
	+ \sqrt{\gamma_d(a)} \right]
	[\text{\rm Lip}(g)]^2,
\end{equation}
uniformly for all real numbers $t\in[0\,,T]$ and $N\ge1$, and
for all Lipschitz-continuous
functions $g:\R\to\R$. Since $\text{\rm Lip}(g)=\|g'\|_{L^\infty(\R)}$,
\eqref{eq:P:1} is equivalent to the assertion that the occupation measure
$\mu_{t,N}$ satisfies a {\it bona fide} 
$L^\infty$-Poincar\'e inequality \cite{Evans,Ledoux}
with Poincar\'e constant
\[
	C_T\inf_{a\in(0,c_d)}\left[
	\tau_d\left( \frac{|\log(1/a)|^{3d/2}}{N^d}\right)
	+ \sqrt{\gamma_d(a)} \right],
\]
uniformly for all $N\ge1$ and $t\in[0\,,T]$.
Now that we have justified the terminology,
let us begin work toward establishing Theorem \ref{th:Poincare}.

The proof of Theorem \ref{th:Poincare}
requires a few technical lemmas, which we develop next.

\begin{lemma}\label{lem:tau_d}
	For all $d\ge1$, $\tau_d$ is monotonically increasing, $\tau_d(ca)\le c\tau_d(a)$
	for all $a>0$ and $c > 1$, and
	\[
		\lim_{a\to0^+}\tau_d(a)=0,\quad\text{and}\quad
		\tau_d(b)\ge b\qquad\text{for all $b\ge0$}.
	\]
\end{lemma}

\begin{proof}
	The monotonicity of $\tau_d$ and the multiplicative property of $\tau_d$ are obvious.
	Clearly, $\tau_d(b) \ge \inf_{r>1}(br^d)=b$ for all $b\ge0.$ Furthermore, we may first
	choose and fix some $r>1$ to see that
	$\tau_d(a)\le ar^d+\|h\|_{L^2(\bm{B}_r^c)}^2\to\|h\|_{L^2(\bm{B}_r^c)}^2$,
	as $a\to0^+$. Then let $r\to\infty$ in order to see that $\tau_d$ vanishes continuously at $0$.
\end{proof}


The following noise-localization lemma estimates how much
the solution to \eqref{SHE} is perturbed if we replace $h$ by
its localization $h_r$. This amounts to replacing $f$ by
$h_r*\tilde{h}_r$. For this reason, we denote the solution to
\eqref{SHE} by $u^{(h)}$ (instead of $u$), so that $u^{(h_r)}$
also makes sense but in that case the noise $\eta=\eta^{(h)}$ in \eqref{SHE}
--- see also \eqref{eta^h} --- is replaced
by $\eta^{(h_r)}$ etc.

\begin{lemma}\label{lem:u-u^h_r}
	For every $k\ge2$ there exists a real number $c=c(k\,,\sigma\,,h)>0$ such that
	\[
		\sup_{x\in\R^d}
		\E\left( \left| u^{(h)}(t\,,x) - u^{(h_r)}(t\,,x) \right|^k\right)
		\le c\,\e^{ct}\left(\int_{\|w\|>r}|h(w)|^2\,\d w\right)^{k/2},
	\]
	uniformly for every $t,r>0$.
\end{lemma}

\begin{proof}
	Throughout, we choose and fix some $\varepsilon>0$
	and set $\beta$ the same as in \eqref{def:beta}:
	\begin{equation}\label{b}
		\beta \ge 1\vee\Lambda_h\left(\frac{2(1-\varepsilon)^2}{[z_k\text{\rm
		Lip}(\sigma)]^2}\right).
	\end{equation}
	Now, we write for every $t>0$ and $x\in\R^d$,
	\begin{align*}
		\left\| u^{(h)}(t\,,x) - u^{(h_r)}(t\,,x) \right\|_k &=
			\left\| \left( \bm{p}\circledast \sigma( u^{(h)})\eta^{(h)}\right)(t\,,x)
			-  \left(\bm{p}\circledast \sigma( u^{(h_r)})\eta^{(h_r)} \right)(t\,,x) \right\|_k\\
		&\le \mathcal{Q}_1(t\,,x) + \mathcal{Q}_2(t\,,x),
	\end{align*}
	where
	\begin{align*}
		\mathcal{Q}_1(t\,,x) &:= \left\| \left( \bm{p}\circledast \sigma( u^{(h)})\eta^{(h)}\right)(t\,,x)
			- \left( \bm{p}\circledast \sigma( u^{(h)})\eta^{(h_r)}\right)(t\,,x)\right\|_k,\\
		\mathcal{Q}_2(t\,,x) &:= \left\| \left( \bm{p}\circledast \left[\sigma( u^{(h)})
			-\sigma( u^{(h_r)})\right]\eta^{(h_r)}\right)(t\,,x) \right\|_k.
	\end{align*}
	Appeal to Lemma \ref{lem:localize:noise} in order to see that, uniformly for all $t>0$,
	\begin{align*}
		\sup_{x\in\R^d}
			\mathcal{Q}_1(t\,,x) &\le \frac{z_k\e^{\beta t}
			\|h\|_{L^2(\mathbb{B}_r^c)}}{\sqrt{2\beta}}
			\mathcal{N}_{\beta,k}\left(\sigma(u^{(h)})\right)\\
		&\le\frac{z_k\e^{\beta t}}{\varepsilon\sqrt{2\beta}}\left[ \|u_0\|_{L^\infty(\R^d)} +
			\frac{|\sigma(0)|}{\text{\rm Lip}(\sigma)}\right]
			\|h\|_{L^2(\mathbb{B}_r^c)};
	\end{align*}
	see Theorem \ref{th:exist} for the last line.
	
	Next, we may note that
	\begin{equation}\label{LamLam}
		\Lambda_{h_r}(a)\le\Lambda_h(a)
		\qquad\text{for all $a>0$.}
	\end{equation}
	This holds simply because $|h_r(x)|\le|h(x)|$ for all $x\in\R^d$. Among other things,
	it follows from this and \eqref{b} that the same $\beta$ satisfies
	$\beta\ge\Lambda_{h_r}(2(1-\varepsilon)^2/[z_k\text{\rm
	Lip}(\sigma)]^2)$. Thus, we may proceed
	in the same way as in the previous display, and appeal to Lemma \ref{lem:Young} in
	order to see that, for the same $\beta$ as above,
	\begin{align*}
		\mathcal{Q}_2(t\,,x) & \le
			z_k\e^{\beta t}
			\mathcal{N}_{\beta,k}\left( \sigma( u^{(h)})
			-\sigma( u^{(h_r)})\right)\sqrt{\frac12\int_{\R^d} \bm{v}_\beta(x)
		\left(|h_r|*|\tilde{h_r}|\right)(x)\,\d x}\\
		&\le \e^{\beta t}(1 - \varepsilon)
			\mathcal{N}_{\beta,k}\left( u^{(h)} - u^{(h_r)}\right),
	\end{align*}
	uniformly for all $t>0$ and $x\in\R^d$,
	where the second inequality is due to \eqref{E:Lambda-Beta}.
	Combine these bounds with \eqref{b}
	in order to find that
	\begin{align*}
		&\e^{-\beta t}\left\| u^{(h)}(t\,,x) - u^{(h_r)}(t\,,x) \right\|_k \\
		& \hskip1in \le \frac{z_k}{\varepsilon\sqrt{2\beta}}\left[ \|u_0\|_{L^\infty(\R^d)} +
			\frac{|\sigma(0)|}{\text{\rm Lip}(\sigma)}\right]
			\|h\|_{L^2(\mathbb{B}_r^c)} + (1 - \varepsilon)
			\mathcal{N}_{\beta,k}\left( u^{(h)} - u^{(h_r)}\right).
	\end{align*}
	Since the right-hand side does not depend on $(t\,,x)$, we may optimize
	and solve for $\mathcal{N}_{\beta,k}(u^{(h)}-u^{(h_r)})$ in order to see that
	\[
		\mathcal{N}_{\beta,k}( u^{(h)}-u^{(h_r)})
		\le  \frac{z_k}{\varepsilon^2\sqrt{2\beta}}\left[ \|u_0\|_{L^\infty(\R^d)} +
		\frac{|\sigma(0)|}{\text{\rm Lip}(\sigma)}\right]
		\|h\|_{L^2(\mathbb{B}_r^c)},
	\]
	which contains the desired result. As a final remark,
	let us mention that, in order to deduce these facts,
	we need also the \textit{a priori} fact that
	$\mathcal{N}_{\beta,k}(u^{(h)})+\mathcal{N}_{\beta,k}(u^{(h_r)})<\infty$ for
	every $r>0$ and for
	the same $\beta$ as in \eqref{b}. This too follows from Theorem \ref{th:exist}
	and the observation \eqref{LamLam}.
\end{proof}

Next we compare the occupation measure of $u^{(h)}(t)$ to
that of the strongly-localized version $u^{(m,h)}(t)$; see \eqref{eq:u^m,h}
for the notation.

\begin{lemma}\label{lem:g-g:1}
	Choose and fix a real number $T>0$ and an integer $k\ge1$. Then,
	\begin{align*}
		\sup_{N>0}\sup_{t\in[0,T]}
			\E\left(\left| \fint_{[0,N]^d} \prod_{j=1}^k
			g_j\left( u^{(h)}(t\,,x + \zeta^j)\right)\d x -
			\fint_{[0,N]^d} \prod_{j=1}^k g_j
			\left( u_n^{(m,h_r)}(t\,,x + \zeta^j)\right)\d x\right|^2\right)&\\
		\lesssim \int_{\|w\|>r}|h(w)|^2\,\d w
			+\gamma_d\left(m^{d-2}\e^{-m^2/2}\right) +  \e^{-n}&,
	\end{align*}
	uniformly for every $r>0$, $n\in\mathbb{N}$,
	$\zeta^1,\ldots,\zeta^k\in\R^d$,  $m\ge1$, and all
	Lipschitz-continuous functions $g_1,\ldots,g_k:\R\to\R$
	that satisfy \eqref{eq:WLOG}.
\end{lemma}

\begin{proof}
	By the triangle inequality,
	\[
		\left\| \prod_{j=1}^kg_j\left( u^{(h)}(t\,,x+\zeta^j)\right)-
		\prod_{j=1}^kg_j\left( u^{(h_r)}(t\,,x +\zeta^j)
		\right)\right\|_2 \le Q_1+Q_2,
	\]
	where
	\begin{align*}
		Q_1 &:=\left\| \left\{ g_1\left( u^{(h)}(t\,,x+\zeta^1)\right) -
			g_1\left( u^{(h_r)}(t\,,x+\zeta^1) \right)\right\}
			\prod_{j=2}^k g_j\left( u^{(h)}(t\,,x+\zeta^j)\right)\right\|_2,\\
		Q_2&:=\left\| g\left( u^{(h_r)}(t\,,x+\zeta^1)\right)
			\left\{\prod_{j=2}^k g_j\left( u^{(h)}(t\,,x+\zeta^j)\right)-
			\prod_{j=2}^k g_j\left( u^{(h_r)}(t\,,x +\zeta^j)
			\right) \right\} \right\|_2.
	\end{align*}
	Because of \eqref{eq:WLOG}, $|g_j(z)|\le|z|$ for all $z\in\R$ and $j=1,\ldots,k$,
	and hence
	\begin{align*}
		Q_1 &\le \left\| \left| u^{(h)}(t\,,x+\zeta^1) -
			u^{(h_r)}(t\,,x+\zeta^1) \right|\cdot
			\prod_{j=2}^k \left| u^{(h)}(t\,,x+\zeta^j)\right|\right\|_2\\
		&\le \left\|  u^{(h)}(t\,,x+\zeta^1) -
			u^{(h_r)}(t\,,x+\zeta^1) \right\|_{2k}\cdot\prod_{j=2}^k
			\left\| u^{(h)}(t\,,x+\zeta^j) \right\|_{2k}.
	\end{align*}
	We have also used H\"older's inequality in the following form:
	$\|X_1\cdots X_k\|_2\le \prod_{j=1}^k\|X_j\|_{2k}$ for every
	$X_1,\ldots, X_k\in L^{2k}(\P)$. Thanks to
	Theorem \ref{th:exist} and Lemma \ref{lem:u-u^h_r}, we can find a positive number
	$A_1=A_1(\sigma\,,h\,,k)$ such that
	\begin{equation}\label{Q:1}
		Q_1 \le  A_1\e^{A_1t}\|h\|_{L^2(\mathbb{B}_r^c)},
	\end{equation}
	uniformly for all $t,r>0$.
	
	Similarly, there exists a positive number
	$A_2=A_2(\sigma\,,h\,,k)$
	\begin{align*}
		Q_2 &\le\left\| u^{(h_r)}(t\,,x+\zeta^1)\right\|_4
			\left\| \prod_{j=2}^k g_j\left( u^{(h)}(t\,,x+\zeta^j)\right)-
			\prod_{j=2}^k g_j\left( u^{(h_r)}(t\,,x +\zeta^j)
			\right)  \right\|_4\\
		&\le A_2\e^{A_2t}
			\left\| \prod_{j=2}^k g_j\left( u^{(h)}(t\,,x+\zeta^j)\right)-
			\prod_{j=2}^k g_j\left( u^{(h_r)}(t\,,x +\zeta^j)
			\right)  \right\|_4.
	\end{align*}
	We can combine the bounds for $Q_1$ and $Q_2$ in order to find that
	there exists a positive number $B_1=B_1(\sigma\,,h\,,k)$ such that
	\begin{align*}
		&\left\| \prod_{j=1}^kg_j\left( u^{(h)}(t\,,x+\zeta^j)\right)-
			\prod_{j=1}^kg_j\left( u^{(h_r)}(t\,,x +\zeta^j)
			\right)\right\|_2\\
		&\le B_1\e^{B_1t}\left(\|h\|_{L^2(\mathbb{B}_r^c)} +
			\left\| \prod_{j=2}^k g_j\left( u^{(h)}(t\,,x+\zeta^j)\right)-
			\prod_{j=2}^k g_j\left( u^{(h_r)}(t\,,x +\zeta^j)
			\right)  \right\|_4\right).
	\end{align*}
	We repeat the method once more in order to see that
	there exists a positive number $B_1'=B_1'(\sigma\,,h\,,k)$ such that
	\begin{align*}
		&\left\| \prod_{j=1}^kg_j\left( u^{(h)}(t\,,x+\zeta^j)\right)-
			\prod_{j=1}^kg_j\left( u^{(h_r)}(t\,,x +\zeta^j)
			\right)\right\|_2\\
		&\le B_1\e^{B_1t}\left(\|h\|_{L^2(\mathbb{B}_r^c)} +
			B_1'\e^{B_1't}\left\{ \|h\|_{L^2(\mathbb{B}_r^c)}+
			\left\| \prod_{j=3}^k g_j\left( u^{(h)}(t\,,x+\zeta^j)\right)-
			\prod_{j=3}^k g_j\left( u^{(h_r)}(t\,,x +\zeta^j)
			\right)  \right\|_8 \right\}\right)\\
		&\le B_2\e^{B_2t} \left(\|h\|_{L^2(\mathbb{B}_r^c)} +
			\left\| \prod_{j=3}^k g_j\left( u^{(h)}(t\,,x+\zeta^j)\right)-
			\prod_{j=3}^k g_j\left( u^{(h_r)}(t\,,x +\zeta^j)
			\right)  \right\|_8 \right),
	\end{align*}
	where $B_2:= B_1 +B_1' + B_1B_1'$, say. Continue repeating the method to see that
	there exists a real number $A=A(\sigma\,,h\,,k)>0$ such that
	\[
		\left\| \prod_{j=1}^kg_j\left( u^{(h)}(t\,,x+\zeta^j)\right)-
		\prod_{j=1}^kg_j\left( u^{(h_r)}(t\,,x +\zeta^j)
		\right)\right\|_2 \le A\e^{At} \|h\|_{L^2(\mathbb{B}_r^c)},
	\]
	whence also
	\begin{equation}\label{I1}
		\left\| \fint_{[0,N]^d} \prod_{j=1}^kg_j\left( u^{(h)}(t\,,x+\zeta^j)\right)\d x -
		\fint_{[0,N]^d} \prod_{j=1}^kg_j\left( u^{(h_r)}(t\,,x +\zeta^j)
		\right)\d x\right\|_2
		\le B\e^{Bt} \|h\|_{L^2(\mathbb{B}_r^c)},
	\end{equation}
	simultaneously for all $t,r,N>0$, and $x,\zeta^1,\ldots,\zeta^k\in\R^d$.
	
	Since $h\in\mathcal{G}_p(\R^d)$
	implies that $h_r\in\mathcal{G}_p(\R^d)$ for all $r>0$,
	we may argue similarly as we did above, but appeal to Lemma \ref{lem:N(u-um)} (with $h_r$
	in place of $h$) instead
	of Lemma \ref{lem:u-u^h_r}, in order to see that
	there exists a real number $B'=B'(\sigma\,,d\,,h\,,k)>0$ such that
	\begin{equation}\label{I2}\begin{split}
		\left\| \fint_{[0,N]^d} \prod_{j=1}^k
			g_j\left( u^{(h_r)}(t\,,x+\zeta^j)\right)\d x -
			\fint_{[0,N]^d} \prod_{j=1}^k g_j
			\left( u^{(m,h_r)}(t\,,x + \zeta^j)\right)\d x\right\|_2&\\
		\le B' \e^{B't}
			\sqrt{ \gamma_d\left(m^{d-2}\e^{-m^2/2}\right)}&,
	\end{split}\end{equation}
	simultaneously for all $t,r,N>0$, $m\ge1$, and $x,
	\zeta^1,\ldots,\zeta^k\in\R^d$. The validity of this fact
	hinges also on
	the observation that: (1) The constants $c_h$ and $c_{h_r}$
	in Lemma \ref{lem:N(u-um)} satisfy $c_{h_r}\le c_h$. Thus,
	the real numbers $A,A_1,A_2,B_1,B_1',B_2$
	indeed do not depend on the parameter $r$; and
	(2) Eq.\ \eqref{LamLam} ensures that the $\beta$ of
	Lemma \ref{lem:N(u-um)} can be selected independently of $r>0$.
	
	Recall the Picard-iteration approximation $u_n^{(m,h_r)}$ to $u^{(m,h_r)}$
	from the proof of Lemma \ref{lem:u^m,h}. In like manner to the above, we find that
	there exists $B''=B''(\sigma\,,d\,,h\,,k)>0$ such that
	\begin{equation}\label{I3}\begin{split}
		\left\| \fint_{[0,N]^d} \prod_{j=1}^k g_j
			\left( u^{(m,h_r)}(t\,,x + \zeta^j)\right)\d x -
			\fint_{[0,N]^d} \prod_{j=1}^k g_j
			\left( u_n^{(m,h_r)}(t\,,x + \zeta^j)\right)\d x\right\|_2&\\
		\le B'' \e^{B'' t}&\e^{-n/2},
	\end{split}\end{equation}
	thanks to Lemma \ref{lem:u^m,h-u^m,h_n} [with $\varepsilon = 1 -\e^{-1/2}$].
	The preceding holds simultaneously for
	all $t,r,N>0$, $x\in\R^d$, and $n\in\mathbb{Z}_+$. And we emphasize once
	again that $\beta$ can be selected independently of $r>0$ because of our
	earlier observation \eqref{LamLam}.
	The lemma follows upon combining \eqref{I1}, \eqref{I2}, and \eqref{I3}.
\end{proof}

Armed with the preceding, we may compare the variance quantity in the
Poincar\'e inequality for the occupation measure of
$u^{(h)}(t)$ (Theorem \ref{th:Poincare}) to one
for a very strongly-localized version $u_n^{(m,h_r)}$.

\begin{lemma}\label{lem:g-g:2}
	Choose and fix a real number $T>0$ and an integer $k\ge1$. Then,
	\begin{align*}
		&\sup_{N>0}\sup_{t\in[0,T]}
			\left| \Var\left(\fint_{[0,N]^d} \prod_{j=1}^k
			g_j\left( u^{(h)}(t\,,x + \zeta^j)\right)
			\d x \right) - \Var\left(
			\fint_{[0,N]^d} \prod_{j=1}^k g_j\left( u_n^{(m,h_r)}
			(t\,,x+ \zeta^j)\right)\d x\right)\right|\\
		&\hskip3.6in\lesssim \|h\|_{L^2(\mathbb{B}_r^c)}
			+\sqrt{\gamma_d\left(m^{d-2}\e^{-m^2/2}\right)} + \e^{-n/2},
	\end{align*}
	uniformly for every Lipschitz-continuous function $g_1,\ldots,g_k:\R\to\R$
	that satisfy \eqref{eq:WLOG},
	all $T,r>0$,  $m\ge 1$,  $n\in\mathbb{N}$, and $\zeta^1,\ldots,\zeta^k\in\R^d$.
\end{lemma}

\begin{proof}
	In order to simplify the exposition, we will write, in short hand,
	\[
		\mathcal{X} := \fint_{[0,N]^d}
			\prod_{j=1}^k  g_j\left( u^{(h)}(t\,,x + \zeta^j)\right)\d x,\quad
		\mathcal{Y} :=  \fint_{[0,N]^d} \prod_{j=1}^k
			g_j\left( u_n^{(m,h_r)}(t\,,x + \zeta^j)\right)\d x,
	\]
	and suppress the various parameters of the lemma for the time being.
	
	We apply
	Lemma \ref{lem:g-g:1} in order to see that
	\[
		\left| (\E \mathcal{X})^2-(\E \mathcal{Y})^2\right|
		\lesssim\left[ \|h\|_{L^2(\mathbb{B}_r^c)}
		+\sqrt{\gamma_d\left(m^{d-2}\e^{-m^2/2}\right)} + \e^{-n/2}\right]
		\left( |\E \mathcal{X}| + |\E \mathcal{Y}|\right),
	\]
	where the implied constant does not depend on $g_1,\ldots,g_k$
	(except via \eqref{eq:WLOG});
	neither does it depend on $t\in[0\,,T]$, $r,N>0$,  $m\ge 1$, nor
	$\zeta^1,\ldots,\zeta^k\in\R^d$.
	Because of \eqref{eq:WLOG}, $|g_j(z)|\le|z|$ for all $j=1,\ldots,k$ and $z\in\R$,
	and hence Jensen's inequality yields
	\[
		|\E \mathcal{X}| \le \sup_{x\in\R^d}\left\| \prod_{j=1}^k\left|
		u^{(h)}(t\,,x + \zeta^j)\right| \right\|_2
		\le \sup_{x\in\R^d}\left\| u^{(h)}(t\,,x)\right\|_{2k}^k,
	\]
	which is bounded uniformly
	for all $t\in[0\,,T]$ from above by a constant $c=c(T\,,\sigma\,,h\,,k)$;
	see Theorem \ref{th:exist}. Because of our observation
	\eqref{LamLam}, Lemma \ref{lem:u^m,h} ensures that
	$|\E \mathcal{Y}|\le c$,
	uniformly for all $r>0$ and $n,m\ge1$, and for the same real number $c$.
	See also \eqref{N:beta(u^m,h)}. Thus,
	\[
		\left| (\E \mathcal{X})^2-(\E \mathcal{Y})^2\right|
		\lesssim \|h\|_{L^2(\mathbb{B}_r^c)}
		+ \sqrt{\gamma_d\left(m^{d-2}\e^{-m^2/2}\right)}  + \e^{-n/2},
	\]
	where the implied constant does not depend on $g_1,\ldots,g_k$
	(except via \eqref{eq:WLOG});
	neither does it depend on $t\in[0\,,T]$, $r,N>0$, $n,m\ge 1$, nor
	$\zeta^1,\ldots,\zeta^k\in\R^d$.
	
	Next, we appeal to Lemma \ref{lem:g-g:1}
	in order to see that, for the same constant $c$ as above,
	\begin{align*}
		\left| \E(\mathcal{X}^2) - \E(\mathcal{Y}^2)\right|
			&\le\|\mathcal{X}-\mathcal{Y}\|_2\|\mathcal{X}+\mathcal{Y}\|_2
			\le 2c\|\mathcal{X}-\mathcal{Y}\|_2\\
		&\lesssim\|h\|_{L^2(\mathbb{B}_r^c)}
			+ \sqrt{\gamma_d\left(m^{d-2}\e^{-m^2/2}\right)} + \e^{-n/2},
	\end{align*}
	with the same sort of parameter-independence properties as in the
	previous display.	
	The lemma follows from summing the previous two displays.
\end{proof}

Finally, we estimate the variance of the occupation measure of the
strongly-localized random field $u_n^{(m,h_r)}$.

\begin{lemma}\label{lem:g-g:3}
	Choose and fix a real number $T>0$ and an integer $k\ge1$. Then,
	\begin{align*}
		\sup_{t\in[0,T]}
		\left| \Var\left(
		\fint_{[0,N]^d} \prod_{l=1}^k
		g_l\left( u_n^{(m,h_r)}(t\,,x + \zeta^l)\right)\d x\right)\right|
		\lesssim \frac{n^d(r^d+m^d) + \max_{1\le j\le k}\|\zeta^j\|^d}{N^d},
	\end{align*}
	uniformly for all real numbers $N,n,r,m\ge1$ and all Lipschitz-continuous functions
	$g_1,\ldots,g_k:\R\to\R$ that satisfy \eqref{eq:WLOG}.
\end{lemma}

\begin{proof}
	Since $\eta^{( h_r)}(\d s\,\d x) := \d x\,\int_{\mathbb{B}_r(x)} h (x-y)\,\xi(\d s\,\d y),$
	it follows readily from the properties of the Wiener integral that, for every $\varphi\in C_c(\R^d)$
	and $x_1\ldots,x_k\in\R^d$,
	\[
		\left\{ \int_{(0,t)\times\mathbb{B}_{m\sqrt t}(x_i)}
		\varphi(x_i-y)\,\eta^{(h_r)}(\d s\,\d y)\right\}_{i=1}^k
		\qquad\text{are independent,}
	\]
	provided that $\|x_i-x_j\|>2(r+ m\sqrt t)$ whenever $i\neq j$.
	Since $u_0\equiv1$ and
	\[
		u_{n+1}^{(m,h_r)}(t\,,x) = 1 + \int_{(0,t)\times\mathbb{B}_{m\sqrt t}(x)}
		\bm{p}_{t-s}(x-y)\sigma\left( u_n^{(m,h_r)}(s\,,y)\right)\eta^{(h_r)}(\d s\,\d y),
	\]
	for every $n\ge1$, it follows from induction on $n$, and from the properties of the
	Walsh stochastic integral, that
	\begin{equation}\label{ind}
		\left\{ u_n^{(m,h_r)}(t\,,x_i)\right\}_{i=1}^k
		\quad\text{are independent,}
	\end{equation}
	provided that $\|x_i-x_j\|>2n(r+ m\sqrt t)$ whenever $i\neq j$. For more details, see
	Lemma 5.4 of Conus et al \cite{CJKS2013}.
	
	For all $j\in\mathbb{Z}^d_+$ define
	\[
		I(j) = I(j\,,N) := \left( \prod_{k=1}^d[ j_k\,, j_k+1) \right)\cap [0\,,N]^d
	\]
	and
	\[
		\mathcal{S} := \int_{[0,N]^d} \prod_{l=1}^k g_l
		\left( u_n^{(m,h_r)}(t\,,x + \zeta^l)\right)\d x
		\quad\text{and}\quad
		\mathcal{Y}_j := \int_{I(j)} \prod_{l=1}^k g_l
		\left( u_n^{(m,h_r)}(t\,,x + \zeta^l)\right)\d x.
	\]
	We can then write
	\[
		 \mathcal{S} = \sum_{j\in\mathbb{Z}^d_+:\,0\le |j|\le N} \mathcal{Y}_j
		 \qquad\text{
		where $|j| := \max_{1\le k\le d}|j_k|$ for all $j\in\mathbb{Z}^d_+$.}
	\]
	Thus, we define
	\begin{align}\label{E:L}
		L := 2\max_{1\le j\le k}\|\zeta^j\| + 2\sqrt{d},
	\end{align}
	and write
	\begin{align*}
		\E(\mathcal{S}^2) &= \mathop{\sum\sum}\limits_{\substack{%
			i,j\in\mathbb{Z}^d_+:\,0\le |i|, |j|\le N\\
			\|i-j\|\le 2n(r+m\sqrt t)+L}}\E(\mathcal{Y}_i\mathcal{Y}_j)
			+ \mathop{\sum\sum}\limits_{\substack{%
			i,j\in\mathbb{Z}^d_+:\,0\le |i|, |j|\le N\\
			\|i-j\|> 2n(r+m\sqrt t)+L}}\E(\mathcal{Y}_i)\E(\mathcal{Y}_j)\\
		&\le 2\mathop{\sum\sum}\limits_{\substack{%
			i,j\in\mathbb{Z}^d_+:\,0\le |i|, |j|\le N\\
			\|i-j\|\le 2n(r+m\sqrt t)+L}}\|\mathcal{Y}_i\|_2\|\mathcal{Y}_j\|_2 +
			\left| \E(\mathcal{S})\right|^2,
	\end{align*}
	thanks to \eqref{ind} and the Cauchy--Schwarz inequality.
	Therefore,  uniformly for all $t\in[0\,,T]$,
	\begin{equation}\label{VAR(S)}\begin{split}
		\Var(\mathcal{S}) &\le 2\max_{0\le |j|\le N}\E\left(|\mathcal{Y}_j|^2\right)
			\: N^d\: \left|\left\{ i\in\mathbb{Z}^d_+:\, 0\le |i|\le N\wedge
			\left(2n(r+m\sqrt t) +L\right)\right\}\right|\\
		&\lesssim \max_{0\le |j|\le N}\E\left(|\mathcal{Y}_j|^2\right)
			\: N^d \: \left\{ n^d\left( r^d + m^d\right) + L^d\right\},
	\end{split}\end{equation}
	where the implied constant depends only on $(d\,,T)$, provided additionally that
	$L+2n(r+m\sqrt T)\ge1$. This condition certainly holds for all $n,m,r\ge1$,
	since $L\ge1$.
	Because of \eqref{eq:WLOG}, $|g_l(w)|\le |w|$ for all $l=1,\ldots,k$ and $w\in\R$. Therefore,
	elementary considerations now yield the following:
	\begin{align*}
		\sup_{j\in\mathbb{Z}^d_+}\E\left(|\mathcal{Y}_j|^2\right)
			&=\sup_{j\in\mathbb{Z}^d_+}  \int_{I(j)}\d x\int_{I(j)}\d y\
			\E\left[\prod_{l=1}^k g_l
			\left( u_n^{(m,h_r)}(t\,,x + \zeta^l)\right)
			\prod_{l'=1}^k g_{l'}\left( u_n^{(m,h_r)}(t\,,y + \zeta^{l'})\right)\right]\\
		&\le\sup_{x\in\R^d}
			\E\left(\left| u_n^{(m,h_r)}(t\,,x)\right|^{2k}\right),
	\end{align*}
	provided additionally that $N\ge1$. Therefore, \eqref{N:beta(u^m,h)} implies that
	\[
		\sup_{t\in[0,T]}\sup_{j\in\mathbb{Z}^d_+}\E\left(|\mathcal{Y}_j|^2\right)<\infty,
	\]
	which has the desired result, thanks to \eqref{VAR(S)}.
\end{proof}

Finally, we can present the main result of this section.

\begin{theorem}\label{th:full}
	Choose and fix an arbitrary real number $T>0$ and an integer $k\ge1$. Then,
	\begin{align*}
		\sup_{t\in[0,T]} &\Var\left(\fint_{[0,N]^d} \prod_{l=1}^k g_l(u(t\,,x
			+ \zeta^l))\,\d x\right)\\
		&\hskip1.7in\lesssim \inf_{m\ge 1}\left[
			\tau_d\left( \frac{m^{3d}+\max_{1\le j\le k}\|\zeta^j\|^d}{N^d}\right)
			+\sqrt{\gamma_d\left(m^{d-2}\e^{-m^2/2}\right)}\right],
	\end{align*}
	where the implied constant depends neither on
	$\zeta^1,\ldots,\zeta^k\in\R^d$ nor on $N\ge1$,
	and where $\gamma_d$ and $\tau_d$  were  defined respectively in
	\eqref{gamma_d} and \eqref{tau_d}.
	In particular,
	\begin{equation}\label{eq:Var->0}
		\lim_{N\to\infty}\Var\left(\fint_{[0,N]^d} \prod_{l=1}^k g_l(u(t\,,x
		+ \zeta^l))\,\d x\right)=0.
	\end{equation}
\end{theorem}

\begin{proof}
	Let $L$ be the constant as chosen as was in \eqref{E:L}.
	We assemble Lemmas
	\ref{lem:g-g:2} and \ref{lem:g-g:3} in order to see that
	\begin{align*}
		\sup_{t\in[0,T]} &\left| \Var\left(
			\fint_{[0,N]^d} \prod_{l=1}^k
			g_l\left( u(t\,,x + \zeta^l)\right)\d x\right)\right|\\
		&\hskip1.7in\lesssim \frac{n^d(r^d+m^d) + L^d}{N^d}
			+ \|h\|_{L^2(\mathbb{B}_r^c)}
			+\sqrt{\gamma_d\left(m^{d-2}\e^{-m^2/2}\right)} + \e^{-n/2},
	\end{align*}
	uniformly for all real numbers $N,n,r,m\ge1$ and all Lipschitz-continuous functions
	$g_1,\ldots,g_k:\R\to\R$ that satisfy \eqref{eq:WLOG}. Optimize over all $r>1$
	to find that
	\begin{align*}
		\sup_{t\in[0,T]}&\left| \Var\left(
			\fint_{[0,N]^d} \prod_{l=1}^k
			g_l\left( u(t\,,x + \zeta^l)\right)\d x\right)\right|\\
		&\hskip1.7in\lesssim \frac{n^dm^d + L^d}{N^d}
			+ \tau_d\left( \frac{n^d}{N^d}\right)
			+\sqrt{\gamma_d\left(m^{d-2}\e^{-m^2/2}\right)} + \e^{-n/2},
	\end{align*}
	uniformly for all real numbers $N,m\ge1$ and all Lipschitz-continuous functions
	$g_1,\ldots,g_k:\R\to\R$ that satisfy \eqref{eq:WLOG}. Next we select $n=2m^2$
	to find that
	\begin{align*}
		\sup_{t\in[0,T]}&\left| \Var\left(
			\fint_{[0,N]^d} \prod_{l=1}^k
			g_l\left( u(t\,,x + \zeta^l)\right)\d x\right)\right|\\
		&\hskip1.7in\lesssim \frac{m^{3d} + L^d}{N^d}
			+ \tau_d\left( \frac{m^{2d}}{N^d}\right)
			+\sqrt{\gamma_d\left(m^{d-2}\e^{-m^2/2}\right)} + \e^{-m^2},
	\end{align*}
	uniformly for all real numbers $N\ge1$ and all Lipschitz-continuous functions
	$g_1,\ldots,g_k:\R\to\R$ that satisfy \eqref{eq:WLOG}. Lemma \ref{lem:tau_d}
	implies that
	\[
		\frac{m^{3d} + L^d}{N^d}
		+ \tau_d\left( \frac{m^{2d}}{N^d}\right) \lesssim
		\tau_d\left( \frac{m^{3d} + L^d}{N^d}\right),
	\]
	where the implied constant depends neither on $(m\,,N)\in[1\,,\infty)^2$
	nor on $L$. And Lemma
	\ref{lem:gamma_d} implies that
	\[
		\sqrt{\gamma_d\left(m^{d-2}\e^{-m^2/2}\right)} + \e^{-m^2}
		\asymp \sqrt{\gamma_d\left(m^{d-2}\e^{-m^2/2}\right)},
	\]
	where the implied constant does not depend on $m\ge1$.
	Combine these bounds in order to obtain the inequality of the theorem.
	That inequality
	and Lemma \ref{lem:gamma_d} together imply \eqref{eq:Var->0}, as well.
\end{proof}

\begin{proof}[Proof of Theorem \ref{th:Poincare}]
	We apply Theorem \ref{th:full}, and change variables in the
	``$\inf_{m > 1}$'' as follows:
	\[
		a := m^{d-2}\e^{-m^2/2}.
	\]
	As $m$ ranges over $[1\,,\infty)$, the variable
	$a$ ranges over $(0\,, c_d)$, where
	\[
		c_d := \max_{m\ge 1}  \left[ m^{d-2}\e^{-m^2/2}\right]; 
	\]
	see \eqref{c_d}.
	Since  $\log(1/a) = \frac12m^2 - (d -2)\log m$, there exists a constant $c = c(d) > 1$ such that
	$m^2 \leq c\,\log(1/a)$ for all $m > 1$.
	And because $\tau_d(ca)\le c\tau_d(a)$ for all $a>0$ and $c > 1$ (see Lemma
	\ref{lem:tau_d}), this yields \eqref{eq:Poincare}.
\end{proof}

\begin{proof}[Proof of Theorem \ref{th:main:intro}]
	The stationarity of $u$ was proved in Lemma \ref{lem:stat}, and spatial ergodicity
	follows from \eqref{eq:Var->0} and Lemma \ref{lem:Var:erg}.
\end{proof}

Finally, we verify Corollaries \ref{co:main:intro} and \ref{co:Poincare:intro}.
The proof is elementary. We include it here
however since the proof depends crucially on careful computation of the various exponents
in \eqref{h:1}--\eqref{exponent:alpha}, \eqref{exponent:A}--\eqref{exponent:B} below.

\begin{proof}[Proof of Corollary \ref{co:main:intro}]
	If $h\in L^2(\R^d)$ then we set $p=q=2$ to see that
	$h\in L^p_{\text{\it loc}}(\R^d)$ and
	\[
		\int_0^1\left( \|h\|_{L^p(\mathbb{B}_r)} \|h\|_{L^q(\mathbb{B}_r^c)}
		+ \|h\|_{L^2(\mathbb{B}_r^c)}^2\right)\omega_d(r)\,\d r
		\le 2\|h\|_{L^2(\R^d)}^2\int_0^1\omega_d(r)\,\d r,
	\]
	so that \eqref{cond:omega} holds thanks to the local integrability of $\omega_d$.
	Thus, it remains to assume that \eqref{cond:co:main:intro} holds. In that case,
	we appeal to \eqref{cond:co:main:intro} and integrate in spherical coordinates
	in order to see that
	\[
		\int_{\mathbb{B}_r}|h(x)|^p\,\d x \lesssim
		\int_0^r s^{d-1 - p(d+\alpha)/2}\,\d s\quad\text{%
		simultaneously for every $r\in(0\,,1)$.}
	\]
	Hence,
	\[
		h\in L^p_{\textit{\it loc}}(\R^d)\quad\text{iff}\quad
		p<\frac{2d}{d+\alpha}.
	\]
	Since $\alpha<d$, it follows that $2d/(d+\alpha)>1$
	and hence $h\in L^p_{\text{\it loc}}(\R^d)$ for every $p$ between
	$1$ and $2d/(d+\alpha)$.
	For every such $p$, \eqref{cond:co:main:intro} ensures that
	\begin{equation}\label{h:1}
		\|h\|_{L^p(\mathbb{B}_r)} \lesssim r^{(d/p)-(d+\alpha)/2}\quad\text{%
		simultaneously for every $r\in(0\,,1)$.}
	\end{equation}
	Choose one such $p$ and define $q:=p/(p-1)$, so that $p^{-1}+q^{-1}=1$.
	Eq.\ \eqref{cond:co:main:intro} implies that, for every $r\in(0\,,1)$,
	\begin{align*}
		\int_{\mathbb{B}_r^c} |h(x)|^q\,\d x
			&\le \int_{r<\|x\|<1}|h(x)|^q\,\d x + \int_{\|x\|>1}|h(x)|^q\,\d x\\
		&\lesssim\int_r^1 t^{d-1-q(d+\alpha)/2}\,\d t + \int_1^\infty t^{d-1-q(d+\beta)/2}\,\d t,
	\end{align*}
	where the implied constants do not depend on $r\in(0\,,1)$.
	The first integral is convergent regardless of the choice of $p$ (hence also $q$). The second integral
	converges iff
	\begin{equation}\label{qd}
		q > \frac{2d}{d+\beta},
	\end{equation}
	which can certainly be arranged if $p$ were chosen sufficiently
	close to $1$.\footnote{%
		To be concrete, we may select $1< p < d/(d-1)$
		to ensure that $q>d$, so that \eqref{qd} holds.
		}
	Choose and fix $p>1$ sufficiently close to $1$ in order to ensure that
	\eqref{qd} holds, whence
	\begin{equation}\label{h:2}
		\|h\|_{L^q(\mathbb{B}_r^c)} \lesssim r^{(d/q)-(d+\alpha)/2}\quad\text{%
		simultaneously for every $r\in(0\,,1)$.}
	\end{equation}
	Finally, we may repeat
	the preceding with $q$ replaced everywhere with $2$ in order to see that
	\begin{equation}\label{h:3}
		\|h\|_{L^2(\mathbb{B}_r^c)}
		\lesssim r^{-\alpha/2}\quad\text{%
		simultaneously for every $r\in(0\,,1)$.}
	\end{equation}
	We may now combine \eqref{h:1}, \eqref{h:2}, and \eqref{h:3}
	in order to see that,
	\begin{align}\label{exponent:alpha}
		\|h\|_{L^p(\mathbb{B}_r)}\|h\|_{L^q(\mathbb{B}_r^c)} +
		\|h\|_{L^2(\mathbb{B}_r^c)}^2 \lesssim r^{-\alpha}\quad\text{%
		simultaneously for every $r\in(0\,,1)$.}
	\end{align}
	Because $\alpha<2\wedge d$, it follows that
	$h\in\mathcal{G}_p(\R^d)$ for all $p$ sufficiently close to $1$.
\end{proof}

\begin{proof}[Proof of Corollary \ref{co:Poincare:intro}]
    It is clear from Theorem \ref{th:full} that
    \begin{equation}\label{eq:Poincare_1}
		\sup_{t\in[0,T]} \Var\left(\fint_{[0,N]^d} g(u(t\,,x))\,\d x\right)
		\lesssim \inf_{m\ge 1}\left[
		\tau_d\left( \frac{m^{3d}}{N^d}\right)
		+\sqrt{\gamma_d\left(m^{d-2}\e^{-m^2/2}\right)}\right]
		[\text{\rm Lip}(g)]^2
	\end{equation}
	uniformly for all real numbers $N>1$ and all Lipschitz-continuous functions $g:\R\to\R$.
	The second inequality in \eqref{cond:co:main:intro} implies that
	$\int_{\|w\| > r}|h(w)|^2\d w \lesssim r^{-\beta}$
	simultaneously for every $r\in(1\,, \infty)$.
	Hence, by the definition of $\tau_d$ in \eqref{tau_d},
	\begin{equation}\label{exponent:A}
		\tau_d(a) \lesssim \inf_{r > 1} \left[ar^d + r^{-\beta} \right] \asymp a^{A}
		\qquad\text{%
		simultaneously for every $a\in[0\,, \beta/d)$,}
	\end{equation}
	where  $A:= \beta/(d + \beta) \in (0\,, 1).$
	Moreover,  $\bar{f}(r) \lesssim r^{-\alpha}$
	simultaneously for every $r\in(0\,,1)$, thanks to \eqref{exponent:alpha} and \eqref{eq:PD}.
	Thus,
	\begin{align*}
		\int_0^s\bar{f}(r)\omega_d(r)\d r \lesssim \int_0^sr^{-\alpha}\omega_d(r)\d r
		\asymp
		\begin{cases}
			s^{1 - \alpha} & \text{if}\quad d = 1, \\
			s^{2 - \alpha}\log_+(1/s) & \text{if}\quad d = 2, \\
			s^{2 - \alpha} & \text{if} \quad d \ge 3,
		\end{cases}
	\end{align*}
	uniformly for all $s \in (0\,, 1)$. Since
	$\bar{\bm{v}}_1(s) \asymp  s^{-d + 1}\omega_d(s)$
	uniformly for all $s \in (0\,, 1)$ --- see \eqref{pot:bound} and \eqref{barbar} ---
	we see from the definition of $\gamma_d$ in \eqref{gamma_d} that 
	\begin{align*}
		\gamma_d(t) \lesssim
		\begin{cases}
			t & \text{if}\quad d = 1, \\
			t^{\nu}\log_+(1/t) & \text{if}\quad d = 2, \\
			t^{\nu} & \text{if} \quad d \ge 3,
		\end{cases}
	\end{align*}
	uniformly for all $t \in (0\,,1)$, where
	$\nu:= (2 - \alpha)/(d - \alpha)$.
	Therefore, there exists $B = B(d\,, \alpha) > 0$ such that
	\begin{equation}\label{exponent:B}
		\gamma_d(t) \lesssim t^B\qquad \text{for all $t \in (0\,,1)$}.
	\end{equation}

	Choose and fix $q > Ad/B$, and let $m:=(8q\log N)^{1/2}$.
	There exists a constant $\tilde{c} = \tilde{c}(d\,, \beta\,, \alpha)$ large enough
	to ensure that $m^{3d}/N^d < \beta/d$ for all $N > \tilde{c}$.
	It then follows from \eqref{eq:Poincare_1}--\eqref{exponent:B} that
	\begin{align*}
		\sup_{t\in[0,T]} \Var\left(\fint_{[0,N]^d} g(u(t\,,x))\,\d x\right) &\lesssim \left[
			\left(\frac{m^{3}}{N}\right)^{Ad}
			+m^{(d-2)B/2}\e^{-Bm^2/4}\right]
			[\text{\rm Lip}(g)]^2 \\
		& \lesssim  \left[
			\left(\frac{m^{3}}{N}\right)^{Ad}
			+\e^{-Bm^2/8}\right]
			[\text{\rm Lip}(g)]^2 \\
		& \asymp   \left[
			\left(\frac{(\log N)^{3/2}}{N}\right)^{Ad}
			+\left(\frac{1}{N}\right)^{Bq}\right]
			[\text{\rm Lip}(g)]^2 \\
		& \asymp  [\text{\rm Lip}(g)]^2
			\left(\frac{(\log N)^{3/2}}{N}\right)^{d\beta/(d + \beta)},
	\end{align*}
	uniformly for all $N > \tilde{c}$. This completes the proof of Corollary \ref{co:Poincare:intro}.
\end{proof}


\begin{thebibliography}{999}
\bibitem{BertiniCancrini1995}
	Bertini, Lorenzo and  Nicoletta Cancrini (1995).
	The stochastic heat equation: Feynman-Kac formula and intermittence.
	{\it J. Statist.\ Phys.}\ {\bf 78}{\it (5-6)} 1377--1401.
\bibitem{Burkholder} 
	Burkholder, D. L. (1966).
	Martingale transforms.
	{\it Ann.\ Math.\ Statist.}\ {\bf 37} 1494--1504.
\bibitem{BDG} 
	Burkholder, D. L., Davis, B. J., and Gundy, R. F. (1972).
	Integral inequalities for convex functions of operators on martingales.
	In: {\it Proceedings of the Sixth Berkeley Symposium on Mathematical
	Statistics and Probability} {\bf II}, 223--240, University of California Press,
	Berkeley, California.
\bibitem{BG} 
	Burkholder, D. L. and  Gundy, R. F. (1970).
	Extrapolation and interpolation of quasi-linear operators on martingales.
	\emph{Acta Math.}\ \textbf{124} 249--304. 
\bibitem{CardonWeberMillet}
	Cardon-Weber, C. and A. Millet (2004).
	On strongly Petrovski\u\i's parabolic SPDEs in arbitrary 
	dimension and application to the stochastic Cahn-Hilliard equation.
	{\it J. Theoret.\ Probab.}\ {\bf 17}{\it (1)} 1--49.
\bibitem{CarlenKree1991}
	Carlen, Eric and Paul Kr\'ee (1991).
	$L^p$ estimates on iterated stochastic integrals.
	{\it Ann.\ Probab.}\ {\bf 19}{\it(1)} 354--368.
\bibitem{ChenHuNualart2017}
	Chen, Xia, Yaozhong Hu, and David Nulart (2017).
	Spatial asymptotics for the parabolic Anderson model driven by a Gaussian rough noise.
	{\it Electron.\ J. Probab.}\ {\bf 22}{\it (65)} 38 pp.
\bibitem{LCHN17}
    	Chen, Le, Yaozhong Hu, and David Nualart (2017).
    	Two-point correlation function and Feynman-Kac formula for the stochastic heat equation.
    	{\it Potential Anal.}\ {\bf 46}{\it (4)} 779--797.
\bibitem{ChenHuang2019}
	Chen, Le and Jingyu Huang (2019).
	Regularity and strict positivity of densities for the stochastic heat equation on
	$\R^d$.
	Preprint available at \url{https://arxiv.org/abs/1902.02382}.
\bibitem{CKP2019}
	Chen, Le, Davar Khoshnevisan, and Fei Pu (2019).
	A Poincar\'e-type inequality for parabolic SPDEs, with applications.
	Manuscript.
\bibitem{ConusKh2012}
	Conus, Daniel and Davar Khoshnevisan (2012).
	On the existence and position of the farthest peaks
	of a family of stochastic heat and wave equations.
	{\it Probab.\ Theory Rel.\ Fields} {\bf 152}{\it (3-4)} 681--701.
\bibitem{CJK2013}
	Conus, Daniel, Mathew Joseph, and Davar Khoshnevisan, (2013).
	On the chaotic character of the stochastic heat equation,
	before the onset of intermitttency.
	{\it Ann.\ Probab.}\ {\bf 41}{\it(3B)} 2225--2260.
\bibitem{CJKS2013}
	Conus, Daniel, Mathew Joseph, Davar Khoshnevisan, and
	Shang-Yuan Shiu (2013).
	On the chaotic character of the stochastic heat equation, II.
	{\it Probab.\ Theory Related Fields} {\bf 156}{\it (3-4)}
	483--533.
\bibitem{Dalang1999}
	Dalang, Robert C. (1999).
	Extending the martingale measure stochastic integral with
	applications to spatially homogeneous s.p.d.e.'s.
	{\it Electron.\ J. Probab.}\ {\bf 4}{\it (6)}, 29 pp.
\bibitem{DalangFrangos}
	Dalang, Robert C. and N. E. Frangos (1998).
	The stochastic wave equation in two dimensions.
	{\it Ann.\ Probab.}\ {\bf 26}{\it (1)} 187--212.
\bibitem{Davis1976}
	Davis, Burgess (1976).
	On the $L^p$ norms of stochastic integrals and other martingales.
	{\it Duke Math.\ J.} {\bf 43}{\it (4)} 697--704.
\bibitem{Doob}
	Doob, J. L. (1990).
	{\it Stochastic Processes.}
	Reprint of the 1953 original.
	John Wiley \&\ Sons, Inc., New York, viii+654.
\bibitem{Evans}
	Evans, Lawrence C. (1998).
	{\it Partial Differential Equations}.
	American Mathematical Society, Providence, RI.
\bibitem{FoondunKhoshnevisan2009}
	Foondun, Mohammud and Davar Khoshnevisan (2009).
	Intermittence and nonlinear parabolic stochastic partial
	differential equations.
	{\it Electron.\ J. Probab.}\ {\bf 14}{\it (21)} 548--568.
\bibitem{FoondunKhoshnevisan2013}
	Foondun, Mohammud and Davar Khoshnevisan (2013).
	On the stochastic heat equation with spatially-colored random forcing.
	{\it Trans.\ Amer.\ Math.\ Soc.}\ {\bf 365}{\it (1)} 409--458.
\bibitem{Hawkes}
	Hawkes, John (1979).
	Potential theory of L\'evy processes.
	{\it Proc.\ London Math.\ Soc.}\
	{\it (3)}{\bf 38}, no.\ 2, 335--352.
\bibitem{Hawkes1984}
	Hawkes, John (1984).
	Some geometric aspects of potential theory.
	In: {\it Stochastic Analysis and Applications} (Swansea, 1983), 130--154, 
	Lecture Notes in Math.\ {\bf 1095} Springer, Berlin.
\bibitem{HuHuangLeNualartTindel2017}
	Hu, Yaozhong, Jingyu, Huang, Khoa L\^e, David Nualart, and Samy Tindel (2017).
	Stochastic heat equation with rough dependence in space.
	{\it Ann.\ Probab.}\ {\bf 45}{\it (6B)} 4561--4616.
\bibitem{HuHuangNualartTindel2015}
	Hu, Yaozhong,  Jingyu Huang, David Nualart, and Samy Tindel (2015).
	Stochastic heat equations with general multiplicative Gaussian noises:
	H\"older continuity and intermittency.
	{\it Electron.\ J. Probab.}\ {\bf 20}{\it (55)} 50 pp.
\bibitem{HuangLeNualart2017}
	Huang, Jingyu, Khoa L\^e, and David Nualart, (2017).
	Large time asymptotics for the parabolic Anderson model driven by
	space and time correlated noise.
	{\it Stoch.\ Partial Differ.\ Equ.\ Anal.\ Comput.}\
	{\bf 5}{\it (4)} 614--651.
\bibitem{HuangNualartViitasaari2018}
	Huang, Jingyu, David Nualart, and Lauri Viitasaari (2018).
	A central limit theorem for the stochastic heat equation.
	Preprint available at \url{https://arxiv.org/abs/1810.09492}.
\bibitem{HuangNualartViitasaariZheng2019}
	Huang, Jingyu, David Nualart, Lauri Viitasaari, and Guangqu Zheng (2019).
	Gaussian fluctuations for the stochastic heat equation with colored noise.
	Preprint available at \url{https://arxiv.org/abs/1903.02509}
\bibitem{Kahane}
	Kahane, Jean-Pierre (1985).
	{\it Some Random Series of Functions.}
	Second edition. Cambridge University Press, Cambridge.
\bibitem{KZ}
	Karczewska, Anna and Jerzy Zabczyk (2001).
	A note on stochastic wave equations. 
	In: {\it Evolution Equations and Their Applications in Physical and
	Life Sciences} (Bad Herrenalb, 1998), 501--511, 
	Lecture Notes in Pure and Appl.\ Math.\ {\bf 215} Dekker, NY. 
\bibitem{KhCBMS}
	Khoshnevisan, Davar (2014).
	{\it Analysis of Stochastic Partial Differential Equations.}
	CBMS Regional Conference Series in Mathematics, 119.
	American Mathematical Society, Providence, RI, 2014. viii+116.
\bibitem{KhMPP}
	Khoshnevisan, Davar (2002).
	{\it Multiparameter Processes.}
	Springer-Verlag, New York, xx+584.
\bibitem{Ledoux}
	Ledoux, Michel (2001).
	{\it The Concentration of Measure Phenomenon.}
	American Mathematical Society, Providence, RI.
\bibitem{MilletSanz}
	Millet, Annie and Marta Sanz-Sol\'e (1999).
	A stochastic wave equation in two space dimension: smoothness of the law.
	{\it Ann.\ Probab.}\ {\bf27}{\it (2)} 803--844.
\bibitem{Peszat2002}
	Peszat, Szymon (2002).
	The Cauchy problem for a nonlinear stochastic wave equation in any dimension.
	{\it J. Evol.\ Equ.}\ {\bf 2}{\it (3)} 383--394.
\bibitem{PeszatZabczyk2000}
	Peszat, Szymon and Jerzy Zabczyk (2000).
	Nonlinear stochastic wave and heat equations.
	{\it Probab.\ Theory Related Fields} {\bf 116}{\it (3)} 421--443.
\bibitem{Peterson}
	Peterson, Karl (1990).
	{\it Ergodic Theory}.
	Cambridge University Press, Cambridge.
\bibitem{Stein}
	Stein, Elias M. (1970).
	{\it Singular Integrals and Differentiability Properties of
	Functions.}  Princeton University Press, Princeton, N.J.
\bibitem{Walsh}
	Walsh, John B. (1986).
	{\it An Introduction to Stochastic Partial Differential Equations.}
	\`Ecole d'\'et\'e de probabilit\'es de Saint-Flour, XIV-1984, 265--439,
	In: {\it Lecture Notes in Math.}\ {\bf 1180}, Springer, Berlin.
\end{thebibliography}
\end{document}